\documentclass[12pt]{article}
\usepackage[utf8]{inputenc}

\title{Optimal Deep Neural Network Approximation for Korobov Functions with respect to Sobolev Norms}

\usepackage{PRIMEarxiv}
\usepackage[utf8]{inputenc} % allow utf-8 input
\usepackage[T1]{fontenc}    % use 8-bit T1 fonts
\usepackage{hyperref}       % hyperlinks
\usepackage{url}            % simple URL typesetting
\usepackage{booktabs}       % professional-quality tables      % blackboard math symbols
\usepackage{nicefrac}       % compact symbols for 1/2, etc.
\usepackage{microtype}      % microtypography
\usepackage{lipsum}
\usepackage{fancyhdr}       % header
\usepackage{graphicx}       % graphics
\graphicspath{{media/}}% organize your images and other figures under media/ folder
\usepackage[numbers]{natbib}
\usepackage{amsmath}
\usepackage{epsfig,amsthm}
\usepackage{amssymb,amsfonts}
\usepackage{dsfont}
\usepackage{subfigure}
\usepackage{indentfirst}
\usepackage{mathrsfs}
\usepackage{color}
\usepackage{zymacros}
\usepackage[ruled,linesnumbered]{algorithm2e} 
\author{%
	Yahong Yang \\
	Department of Mathematics\\
	The Pennsylvania State University, University Park\\State College, PA 16802, USA \\
	\texttt{yxy5498@psu.edu}
	\And
	Yulong Lu\\
	School of Mathematics\\ University of Minnesota, Twin Cities \\Minneapolis, MN  55455, USA\\
	\texttt{yulonglu@umn.edu}}
\begin{document}
	
\maketitle
\begin{abstract}
This paper establishes the nearly optimal rate of approximation for deep neural networks (DNNs) when applied to Korobov functions, effectively overcoming the curse of dimensionality. The approximation results presented in this paper are measured with respect to $L_p$ norms and $H^1$ norms. Our achieved approximation rate demonstrates a remarkable \textit{super-convergence} rate, outperforming traditional methods and any continuous function approximator. These results are non-asymptotic, providing error bounds that consider both the width and depth of the networks simultaneously.
\end{abstract}

% REQUIRED

\section{Introduction}
Deep neural networks (DNNs) with the rectified linear unit (ReLU) activation function \cite{glorot2011deep} have become increasingly popular in scientific and engineering applications, including image classification \cite{krizhevsky2017imagenet,he2015delving}, regularization \cite{czarnecki2017sobolev}, and dynamic programming \cite{finlay2018lipschitz,werbos1992approximate}. DNNs can often take advantage in approximating functions with special structures and alleviating the curse of dimensionality \cite{poggio2017and,telgarsky2016benefits,yarotsky2017error}. Therefore, finding the proper structure of approximating functions is necessary. Based on the results in \cite{lu2021deep,shen2022optimal,siegel2022optimal,yang2023nearly,yang2023nearlys,he2023optimal}, if we consider the approximation of DNNs in Sobolev spaces $W^{n,p}([0,1]^d)$, the approximation rate measured by the $W^{m,p}$ for $m< n$, $1\le p\le \infty$, and $m,n\in\sN$ is $\fO\left(M^{-\frac{2(m-n)}{d}}\right)$ (up to logarithmic factors). Although this is an optimal rate, it still suffers from the curse of dimensionality. Here $M$ is the number of the parameters in DNNs. The definition of Sobolev spaces is shown as follows:\begin{definition}[Sobolev Space \cite{evans2022partial}]
			Denote $\Omega$ as $[0,1]^d$, $D$ as the weak derivative of a single variable function and $D^{\boldsymbol{\alpha}}=D^{\alpha_1}_1D^{\alpha_2}_2\ldots D^{\alpha_d}_d$ as the partial derivative where $\boldsymbol{\alpha}=[\alpha_{1},\alpha_{2},\ldots,\alpha_d]^T$ and $D_i$ is the derivative in the $i$-th variable. Let $n\in\sN$ and $1\le p\le \infty$. Then we define Sobolev spaces\[W^{n, p}(\Omega):=\left\{f \in L_p(\Omega): D^{\boldsymbol{\alpha}} f \in L_p(\Omega) \text { for all } \boldsymbol{\alpha} \in \sN^d \text { with }|\boldsymbol{\alpha}| \leq n\right\}\] with a norm \[\|f\|_{W^{n, p}(\Omega)}:=\left(\sum_{0 \leq|\alpha| \leq n}\left\|D^{\alpha} f\right\|_{L_p(\Omega)}^p\right)^{1 / p},\] if $p<\infty$, and $\|f\|_{W^{n, \infty}(\Omega)}:=\max_{0 \leq|\alpha| \leq n}\left\|D^{\alpha} f\right\|_{L_\infty(\Omega)}$.

   Furthermore, for $\boldsymbol{f}=(f_1,\ldots,f_d)$, $\boldsymbol{f}\in W^{1,\infty}(\Omega,\sR^d)$ if and only if $ f_i\in W^{1,\infty}(\Omega)$ for each $i=1,2,\ldots,d$ and \[\|\boldsymbol{f}\|_{W^{1,\infty}(\Omega,\sR^d)}:=\max_{i=1,\ldots,d}\{\|f_i\|_{W^{1,\infty}(\Omega)}\}.\] When $p=2$, denote $W^{n,2}(\Omega)$ as $H^n(\Omega)$ for $n\in\sN_+$.
\end{definition}

One of a proper function spaces are widely used is spectral Barron spaces \cite{barron1993universal,bach2017breaking,klusowski2018approximation,lu2021priori,siegel2022sharp,yang2022approximation,chen2023regularity} and representational Barron spaces \cite{weinan2019barron,weinan2019priori,ma2022barron,chen2021representation}. In such kind of spaces, the functions can be approximated by shallow neural networks by $\fO\left(W^{-\frac{1}{2}}\right)$ where $W$ is the number of the parameters in neural networks. However, the space is limit and the theory of Barron spaces in deep neural networks is still needed to be established.

Note that in the Sobolev space approximation results \cite{lu2021deep,shen2022optimal,siegel2022optimal,yang2023nearly,yang2023nearlys,he2023optimal,he2023optimal}, if the function is smooth enough, such as $f(\boldsymbol{x})\in W^{2d,p}([0,1]^d)$, the approximation rate becomes $\fO(M^{-4})$ measured by the $L_p$ norms. However, this space is limited, especially when $d$ is large, as it requires the target function to have $2d$th derivatives in a single direction. This requirement is too strong, limiting the generality of the space. A proper space that maintains the approximation rate without requiring such smoothness is needed. Here, Korobov spaces are such spaces.
\begin{definition}[Korobov Space \cite{bungartz2004sparse,korobov1959approximate,korobov1963number}]
    The Korobov spaces $X^{2,p}(\Omega)$, defined for $2\le p\le +\infty$, are characterized by the following expression:\[X^{2,p}(\Omega)=\left\{f\in L_p(\Omega)\mid f|_{\partial\Omega}=0,D^{\boldsymbol{k}}f\in L_p(\Omega),|\boldsymbol{k}|_\infty\le 2\right\}\] with $|\boldsymbol{k}|_\infty=\max_{1\le j\le d}k_j$ and norm\[|f|_{2,p}=\left\|\frac{\partial^{2d} f}{\partial x_1^2\cdots\partial x_d^2}\right\|_{L_p(\Omega)}.\]
\end{definition}
The approximation of neural networks for Korobov spaces has been considered in \cite{montanelli2019new,mao2022approximation,blanchard2021shallow,suzuki2018adaptivity}. In \cite{montanelli2019new}, they consider the approximation by ReLU-based DNNs with the error measured by $L_\infty$ norms for $X^{2,\infty}$. The error is $\fO(M^{-2})$. In \cite{blanchard2021shallow}, they obtain such a rate by using a smooth activation function and show that this rate is optimal under the assumption of a continuous function approximator. In \cite{suzuki2018adaptivity}, they obtain a similar rate for mix-Besov spaces, where Korobov spaces are also known as mix-Sobolev spaces. In \cite{mao2022approximation}, they consider the approximation of deep convolutional neural networks for $X^{2,p}$ with error $\fO(M^{-2+\frac{1}{p}})$ measured by $L_p$-norms. When $p=\infty$, the order is the same as that in \cite{montanelli2019new,blanchard2021shallow,suzuki2018adaptivity}.

Based on the above references, there are still questions that need to be solved. One is to determine the best approximation rate of Korobov spaces by DNNs. Although in \cite{blanchard2021shallow}, they stated that their result is optimal under the assumptions of continuous function approximator models. However, based on this assumption, the approximation of DNNs does not have an advantage over traditional methods in terms of error. In this paper, we are able to approximate functions in Korobov spaces by DNNs with an error of $\fO(M^{-4})$, which is called \textit{super-convergence} of DNNs \cite{lu2021deep,shen2022optimal,siegel2022optimal,yang2023nearly,yang2023nearlys,he2023optimal}. The proof method in these cases uses the bit-extraction technique introduced in \cite{bartlett1998almost,bartlett2019nearly} to represent piecewise polynomial functions on the fixed regular grid with $M$ cells using only $\fO(\sqrt{M})$ (Proposition \ref{point}). Furthermore, all results in \cite{montanelli2019new,mao2022approximation,blanchard2021shallow,suzuki2018adaptivity} establish the approximation of DNNs with a fixed depth or establish a relationship between the depth and width of DNNs. However, these results cannot reveal the behavior of the freedom of depth in deep neural networks for the approximation task. This limitation hinders our understanding of why deep neural networks are needed as opposed to shallow networks in this context.

Further, the approximation error for DNNs with respect to  Sobolev norms need to be addressed. Sobolev training \cite{czarnecki2017sobolev,son2021sobolev,vlassis2021sobolev} of DNNs has had a significant impact on scientific and engineering fields, including solving partial differential equations \cite{Lagaris1998,weinan2017deep,raissi2019physics,de2022error,lu2022priori,lu2021priori}, operator learning \cite{lu2021learning,liu2022deep}, network compression \cite{sau2016deep}, distillation \cite{hinton2015distilling,rusu2015policy}, regularization \cite{czarnecki2017sobolev}, and dynamic programming \cite{finlay2018lipschitz,werbos1992approximate}, etc. These loss functions enable models to learn DNNs that can approximate the target function with small discrepancies in both magnitude and derivative. If the approximation error of DNNs for Korobov spaces can be established under Sobolev norms, it can provide theoretical support for DNNs to solve a variety of physical problems, such as solving the electronic Schr{\"o}dinger equation in Hilbert spaces \cite{yserentant2004regularity}, which proves that the solution of the electronic Schr{\"o}dinger equation belongs to the mix-Sobolev spaces. All results in \cite{montanelli2019new,mao2022approximation,blanchard2021shallow,suzuki2018adaptivity} do not consider them in Sobolev norms.

In this paper, we address all the questions above. We consider the approximation of functions \(f(\boldsymbol{x}) \in X^{2,\infty}([0,1]^d)\). All the results established in this paper can be easily generalized to \(X^{2,p}([0,1]^d)\) by combining the findings in \cite{mao2022approximation}. We first establish a DNN with depth \(\mathcal{O}(L(\log_2 L)^{3d})\) and width \(\mathcal{O}(L(\log_2 L)^{3d})\) to approximate \(f \in X^{2,\infty}\) with an error of \(\mathcal{O}(N^{-1}L^{-1})\) measured by \(H^1\) norms (Theorem \ref{H1}) and an error of \(\mathcal{O}(N^{-2}L^{-2})\) measured by \(L_p\) norms (Corollary \ref{L2}). These results achieve the optimal outcomes of continuous function approximators \cite{devore1989optimal}, and the results are nonasymptotic, providing error bounds in terms of both the width and depth of networks simultaneously. Then we obtain the \textit{super-convergence} rate for \(f \in X^{2,\infty}\) by using the bit-extraction technique introduced in \cite{bartlett1998almost,bartlett2019nearly}. We establish DNNs with depth \(\mathcal{O}(L(\log_2 L)^{d})\) and width \(\mathcal{O}(L(\log_2 L)^{d+1})\) to approximate \(f \in X^{2,\infty}\) with an error of \(\mathcal{O}(N^{-2}L^{-2}(\log_2 N\log_2 L)^{d-1})\) measured by \(H^1\) norms (Theorem \ref{H1 K}) and an error of \(\mathcal{O}(N^{-4}L^{-4}(\log_2 N\log_2 L)^{d-1})\) measured by \(L_p\) norms (Theorem \ref{L2 K}). Then we show that both our results are nearly optimal (Theorems \ref{Optimality} and \ref{OptimalityH1}) based on the bounds of Vapnik--Chervonenkis dimension of DNNs and their derivatives. These results are better than those in \cite{montanelli2019new,mao2022approximation,blanchard2021shallow,suzuki2018adaptivity} and show the benefit of deep neural networks compared to traditional methods. %Finally, we obtain the generalization error of Korobov spaces measured by the \(H^1\) loss functions (Theorem \ref{gen thm}).

The rest of the paper unfolds as follows: In Section~\ref{preliminaries}, we begin by consolidating the notations, propositions, and lemmas related to DNNs, Sobolev spaces, and Korobov spaces. Subsequently, in Section~\ref{continuouss}, we delve into the establishment of DNN approximation with the error through \(L_p\) norms and \(H^1\) norms. The attained rate mirrors that of the optimal continuous function approximator. Sections~\ref{LopKK} and \ref{hop} are dedicated to unveiling DNNs endowed with \textit{super-convergence} rates for Korobov spaces. %Finally, we obtain the generalization error of Korobov spaces with loss functions defined by Sobolev norms.

\section{Preliminaries}\label{preliminaries}
\subsection{Notations of deep neural networks}
		Let us summarize all basic notations used in the DNNs as follows:
		
		\textbf{1}. Matrices are denoted by bold uppercase letters. For example, $\boldsymbol{A}\in\sR^{m\times n}$ is a real matrix of size $m\times n$ and $\boldsymbol{A}^T$ denotes the transpose of $\boldsymbol{A}$.
		
		\textbf{2}. Vectors are denoted by bold lowercase letters. For example, $\boldsymbol{v}\in\sR^n$ is a column vector of size $n$. Furthermore, denote $\boldsymbol{v}(i)$ as the $i$-th elements of $\boldsymbol{v}$.
		
		\textbf{3}. For a $d$-dimensional multi-index $\boldsymbol{\alpha}=[\alpha_1,\alpha_2,\cdots\alpha_d]\in\sN^d$, we denote several related notations as follows: $(a)~ |\boldsymbol{\alpha}|=\left|\alpha_1\right|+\left|\alpha_2\right|+\cdots+\left|\alpha_d\right|$; $(b)~\boldsymbol{x}^\alpha=x_1^{\alpha_1} x_2^{\alpha_2} \cdots x_d^{\alpha_d},~ \boldsymbol{x}=\left[x_1, x_2, \cdots, x_d\right]^T$; $ (c)~\boldsymbol{\alpha} !=\alpha_{1} ! \alpha_{2} ! \cdots \alpha_{d} !.$
		
		\textbf{4}. Let $B_{r,|\cdot|}(\boldsymbol{x})\subset\sR^d$ be the closed ball with a center $\boldsymbol{x}\in\sR^d$ and a radius $r$ measured by the Euclidean distance. Similarly, $B_{r,\|\cdot\|_{\ell_\infty}}(\boldsymbol{x})\subset\sR^d$ be the closed ball with a center $\boldsymbol{x}\in\sR^d$ and a radius $r$ measured by the $\ell_\infty$-norm.
		
		\textbf{5}. Assume $\boldsymbol{n}\in\sN_+^n$, then $f(\boldsymbol{n})=\fO(g(\boldsymbol{n}))$ means that there exists positive $C$ independent of $\boldsymbol{n},f,g$ such that $f(\boldsymbol{n})\le Cg(\boldsymbol{n})$ when all entries of $\boldsymbol{n}$ go to $+\infty$.
		
		\textbf{6}. Define $\sigma(x)=\max\{0,x\}$. We call the neural networks with activation function $\sigma$ as $\sigma$ neural networks ($\sigma$-NNs). With the abuse of notations, we define $\sigma:\sR^d\to\sR^d$ as $\sigma(\boldsymbol{x})=\left[\begin{array}{c}
			\sigma(x_1) \\
			\vdots \\
			 \sigma(x_d)
		\end{array}\right]$ for any $\boldsymbol{x}=\left[x_1, \cdots, x_d\right]^T \in\sR^d$.
		
		\textbf{7}. Define $L,N\in\sN_+$, $N_0=d$ and $N_{L+1}=1$, $N_i\in\sN_+$ for $i=1,2,\ldots,L$, then a $\sigma$-NN $\phi$ with the width $N$ and depth $L$ can be described as follows:\[\boldsymbol{x}=\tilde{\boldsymbol{h}}_0 \stackrel{W_1, b_1}{\longrightarrow} \boldsymbol{h}_1 \stackrel{\sigma}{\longrightarrow} \tilde{\boldsymbol{h}}_1 \ldots \stackrel{W_L, b_L}{\longrightarrow} \boldsymbol{h}_L \stackrel{\sigma}{\longrightarrow} \tilde{\boldsymbol{h}}_L \stackrel{W_{L+1}, b_{L+1}}{\longrightarrow} \phi(\boldsymbol{x})=\boldsymbol{h}_{L+1},\] where $\boldsymbol{W}_i\in\sR^{N_i\times N_{i-1}}$ and $\boldsymbol{b}_i\in\sR^{N_i}$ are the weight matrix and the bias vector in the $i$-th linear transform in $\phi$, respectively, i.e., $\boldsymbol{h}_i:=\boldsymbol{W}_i \tilde{\boldsymbol{h}}_{i-1}+\boldsymbol{b}_i, ~\text { for } i=1, \ldots, L+1$ and $\tilde{\boldsymbol{h}}_i=\sigma\left(\boldsymbol{h}_i\right),\text{ for }i=1, \ldots, L.$ In this paper, an DNN with the width $N$ and depth $L$, means
		(a) The maximum width of this DNN for all hidden layers less than or equal to $N$.
		(b) The number of hidden layers of this DNN less than or equal to $L$.

\subsection{Korobov spaces}

In this paper, our focus centers on investigating the efficacy of deep networks in approximating functions within Korobov spaces, gauged through the measurement of Sobolev norms.

The approach employed for approximating functions within the Korobov space relies on sparse grids, as introduced in \cite{bungartz2004sparse}. For any $f\in X^{2,p}(\Omega)$, the representation takes the form: \[f(\boldsymbol{x})=\sum_{\boldsymbol{l}}\sum_{\boldsymbol{i}\in\boldsymbol{i}_{\boldsymbol{l}}}v_{\boldsymbol{l},\boldsymbol{i}}\phi_{\boldsymbol{l},\boldsymbol{i}}(\boldsymbol{x})\] where \begin{equation}
\boldsymbol{i}_{\boldsymbol{l}}:=\left\{\boldsymbol{i} \in\sN^d: \boldsymbol{1} \leq \boldsymbol{i} \leq 2^{\boldsymbol{l}}-\boldsymbol{1}, i_j \text{ odd for all } 1 \leq j \leq d\right\}.
\end{equation} The basis function $\phi_{\boldsymbol{l},\boldsymbol{i}}(\boldsymbol{x})$ is constructed using hat functions and grid points: \[\boldsymbol{x}_{\boldsymbol{l},\boldsymbol{i}}=(x_{l_1,i_1},\cdots,x_{l_d,i_d}):=\boldsymbol{i}\cdot 2^{-\boldsymbol{l}}=:\boldsymbol{i}\cdot \boldsymbol{h}_{\boldsymbol{l}}=\boldsymbol{i}\cdot (h_{l_1},\cdots,h_{l_d})\] In a piecewise linear setting, the fundamental choice for a 1D basis function is the standard hat function $\phi(x)$, defined as:
$$
\phi(x):= \begin{cases}1-|x|, & \text { if } x \in[-1,1] \\ 0, & \text { otherwise }\end{cases}
$$
The standard hat function $\phi(x)$ can be utilized to generate any $\phi_{l_j, i_j}\left(x_j\right)$ with support $\left[x_{l_j, i_j}-h_{l_j}, x_{l_j, i_j}+h_{l_j}\right]=\left[\left(i_j-1\right) h_{l_j},\left(i_j+\right.\right.$ 1) $h_{l_j}$ ] through dilation and translation:
$$
\phi_{l_j, i_j}\left(x_j\right):=\phi\left(\frac{x_j-i_j \cdot h_{l_j}}{h_{l_j}}\right) .
$$
The resulting 1D basis functions serve as inputs for the tensor product construction, yielding a suitable piecewise $d$-linear basis function at each grid point $\boldsymbol{x}_{\boldsymbol{l}, \boldsymbol{i}}$
$$
\phi_{\boldsymbol{l}, \boldsymbol{i}}(\boldsymbol{x}):=\prod_{j=1}^d \phi_{l_j, i_j}\left(x_j\right)
$$

The following two lemmas pertain to the truncation error in the hierarchical representation of Korobov spaces.

 \begin{lemma}[{\cite[Lemma 3.3]{bungartz2004sparse}}]\label{boundconstant}
      Let $f \in X^{2,\infty}(\Omega)$ be given in its hierarchical representation \[f(\boldsymbol{x})=\sum_{\boldsymbol{l}}\sum_{\boldsymbol{i}\in\boldsymbol{i}_{\boldsymbol{l}}}v_{\boldsymbol{l},\boldsymbol{i}}\phi_{\boldsymbol{l},\boldsymbol{i}}(\boldsymbol{x}).\] Then, the following estimates for the hierarchical coefficients $v_{\boldsymbol{l},\boldsymbol{i}}$ hold:
$$
\left|v_{\boldsymbol{l},\boldsymbol{i}}\right| \leq 2^{-d} \cdot 2^{-\left|\boldsymbol{l}\right|_1} \cdot|f|_{2, \infty}.
$$
  \end{lemma}
  
%Based on the \cite[Lemma 3.13]{bungartz2004sparse}, we have that
\begin{lemma}[{\cite[Lemma 3.13]{bungartz2004sparse}}]\label{err-first}
    Set
$$
f_n^{(1)}(\boldsymbol{x})=\sum_{|\boldsymbol{l}|_1 \leq n+d-1} \sum_{\boldsymbol{i} \in \boldsymbol{i}_{\boldsymbol{l}}} v_{\boldsymbol{l}, \boldsymbol{i}} \phi_{\boldsymbol{l}, \boldsymbol{i}}(\boldsymbol{x}),
$$
and for any $f \in X^{2, p}(\Omega)$, the approximation error satisfies
$$
\left\|f-f_n^{(1)}\right\|_{L_\infty(\Omega)}=\mathcal{O}\left(M^{-2}\left|\log _2 M\right|^{3(d-1)}\right),~\left\|f-f_n^{(1)}\right\|_{H^1(\Omega)}=\mathcal{O}\left(M^{-1}\left|\log _2 M\right|^{(d-1)}\right)
$$
where $M=\fO(2^nn^{d-1})$.
%and, for any accuracy $\epsilon>0$,
%\begin{align}
%\left\|f-f_n^{(1)}\right\|_{L_\infty(\Omega)}=\epsilon \quad& \text { with } \quad %M=\fO\left(\epsilon^{-\frac{1}{2}}\left|\log _2 \epsilon\right|^{\frac{3}{2}(d-1)}\right)\notag\\\left\|f-f_n^{(1)}\right\|_{H^1(\Omega)}=\epsilon \quad& \text { with } \quad M=\fO\left(\epsilon^{-1}\left|\log _2 \epsilon\right|^{d-1}\right)
%\end{align}
\end{lemma}

The remaining challenge within the approximation segment is to represent $\phi_{\boldsymbol{l}, \boldsymbol{i}}(\boldsymbol{x})$ through deep neural networks.

  \subsection{Propositions of Sobolev spaces and ReLU neural networks}

 Before establishing DNNs for approximating functions measured by $L_p$ norms and $H^1$ norms, we require several lemmas and propositions related to DNNs.

The proofs of the following two propositions can be found in \cite{yang2023nearly}. For clarity, we present the proofs here.
\begin{proposition}\label{2prop}
			For any $N,L\in\sN_+$ and $a>0$, there is a $\sigma$-NN $\phi$ with the width $15N$ and depth $2L$ such that $\|\phi\|_{W^{1,\infty}((-a,a)^2)}\le 12a^2$ and \begin{equation}
				\left\|\phi(x,y)-xy\right\|_{W^{1,\infty}((-a,a)^2)}\le 6a^2N^{-L}.
			\end{equation} Furthermore, \begin{equation}\phi(0,y)=\frac{\partial \phi(0,y)}{\partial y}=0,~y\in(-a,a).\label{zero}\end{equation}
		\end{proposition}
		
		\begin{proof}
			We first need to construct a neural network to approximate $x^2$ on $(-1,1)$, and the idea is similar with \cite[Lemma 3.2]{hon2022simultaneous} and \cite[Lemma 5.1]{lu2021deep}. The reason we do not use \cite[Lemma 3.4]{hon2022simultaneous} and \cite[Lemma 4.2]{lu2021deep} directly is that constructing $\phi(x,y)$ by translating a neural network in $W^{1,\infty}[0,1]$ will lose the proposition of $\phi(0.y)=0$. Here we need to define teeth functions $T_i$ on $\widetilde{x}\in [-1,1]$:$$
			T_1(\widetilde{x})= \begin{cases}2 |\widetilde{x}|, & |\widetilde{x}| \leq \frac{1}{2}, \\ 2(1-|\widetilde{x}|), & |\widetilde{x}|>\frac{1}{2},\end{cases}
			$$
			and
			$$
			T_i=T_{i-1} \circ T_1, \quad \text { for } i=2,3, \cdots .
			$$
			Define \[\widetilde{\psi}(\widetilde{x})=\widetilde{x}-\sum_{i=1}^s \frac{T_i(\widetilde{x})}{2^{2 i}},\]According to \cite[Lemma 3.2]{hon2022simultaneous} and \cite[Lemma 5.1]{lu2021deep}, we know $\psi$ is a neural network with the width $5N$ and depth $2L$ such that $\|\widetilde{\psi}(\widetilde{x})\|_{W^{1,\infty}((-1,1))}\le 2$,  $\|\widetilde{\psi}(\widetilde{x})-\widetilde{x}^2\|_{W^{1,\infty}((-1,1))}\le N^{-L}$ and $\psi(0)=0$.
			
			By setting $x=a\widetilde{x}\in(-a,a)$ for $\widetilde{x}\in(-1,1)$, we define \[\psi(x)=a^2 \widetilde{\psi}\left(\frac{x}{a}\right).\]Note that $x^2=a^2 \left(\frac{x}{a}\right)^2$, we have \[\begin{aligned}
				\|\psi(x)-x^2\|_{\mathcal{W}^{1, \infty}\left(-a, a\right)} & =a^2\left\|\widetilde{\psi}\left(\frac{x}{a}\right)-\left(\frac{x}{a} \right)^2\right\|_{\mathcal{W}^{1, \infty}(\left(-a, a\right))} \\
				& \leq a^2 N^{-L},
			\end{aligned}\] and $\psi(0)=0$, which will be used to prove Eq.~(\ref{zero}). 
			
			Then we can construct $\phi(x,y)$ as  \begin{equation}
				\phi(x,y)=2\left[\psi\left(\frac{|x+y|}{2}\right)-\psi\left(\frac{|x|}{2}\right)-\psi\left(\frac{|y|}{2}\right)\right]\label{zero1}\end{equation}where $\phi(x)$ is a neural network with the width $15N$ and depth $2L$ such that $\|\phi\|_{W^{1,\infty}((-a,a)^2)}\le 12a^2$ and \begin{equation}
				\left\|\phi(x,y)-xy\right\|_{W^{1,\infty}((-a,a)^2)}\le 6a^2N^{-L}.
			\end{equation} For the last equation Eq.~(\ref{zero}) is due to $\phi(x,y)$ in the proof can be read as Eq.~(\ref{zero1}) with $\psi(0)=0$.
		\end{proof}

		\begin{proposition}\label{prop1}
			For any $N, L, s \in\sN_+$with $s \geq 2$, there exists a $\sigma$-NN $\phi$ with the width $9(N+1)+s-1$ and depth $14 s(s-1) L$ such that $\|\phi\|_{\mathcal{W}^{1, \infty}((0,1)^s)} \leq 18$ and
			\begin{equation}
				\left\|\phi(\boldsymbol{x})-x_1 x_2 \cdots x_s\right\|_{\mathcal{W}^{1, \infty}((0,1)^s)} \leq 10(s-1)(N+1)^{-7 s L}.\label{mmprod}
			\end{equation}
			Furthermore, for any $i=1,2,\ldots, s$, if $x_i=0$, we will have \begin{equation}
				\phi(x_1,x_2,\ldots,x_{i-1},0,x_{i+1},\ldots,x_s)=\frac{\partial \phi(x_1,x_2,\ldots,x_{i-1},0,x_{i+1},\ldots,x_s)}{\partial x_j}=0, ~i\not=j.\label{equal0}
			\end{equation}
		\end{proposition}
		
		\begin{proof}
			The proof of the first inequality Eq.~(\ref{mmprod}) can be found in \cite[Lemma 3.5]{hon2022simultaneous}. The proof of Eq.~(\ref{equal0}) can be obtained via induction. For $s=2$, based on Proposition \ref{2prop}, we know there is a neural network $\phi_2$ satisfied Eq.~(\ref{equal0}). 
			
			Now assume that for any $i\le n-1$, there is a neural network $\phi_i$ satisfied Eq.~(\ref{equal0}).  $\phi_n$ in \cite{hon2022simultaneous} is constructed as\begin{equation}
				\phi_n(x_1,x_2,\ldots,x_n)=\phi_2(\phi_{n-1}(x_1,x_2,\ldots,x_{n-1}),\sigma(x_n)),
			\end{equation} which satisfies Eq.~(\ref{mmprod}). Then $\phi_n(x_1,x_2,\ldots,x_{i-1},0,x_{i+1},\ldots,x_n)=0$ for any $i=1,2,\ldots,n$. For $i=n$, we have \begin{equation}
				\frac{\phi(x_1,x_2,\ldots,0)}{\partial x_j}=\underbrace{\frac{\partial \phi_2(\phi_{n-1}(x_1,x_2,\ldots,x_{n-1}),0)}{\partial  \phi_{n-1}(x_1,x_2,\ldots,x_{n-1})}}_{=0\text{, by the property of $\phi_2$.}}\cdot \frac{\partial \phi_{n-1}(x_1,x_2,\ldots,x_{n-1})}{\partial x_j}=0.
			\end{equation}
			
			For $i<n$ and $j<n$, we have \begin{align}
				&\frac{\phi(x_1,x_2,\ldots,x_{i-1},0,x_{i+1},\ldots,x_n)}{\partial x_j}\notag\\=&\frac{\partial \phi_2(\phi_{n-1}(x_1,x_2,\ldots,x_{i-1},0,x_{i+1},\ldots,x_{n-1}),\sigma(x_n))}{\partial \phi_{n-1}(x_1,\ldots,0,x_{i+1},\ldots,x_{n-1})}\cdot \underbrace{\frac{\partial \phi_{n-1}(x_1,\ldots,0,x_{i+1},\ldots,x_{n-1})}{\partial x_j}}_{=0\text{, via induction.}}=0.
			\end{align}
			
			For $i<n$ and $j=n$, we have 
   \begin{align}
				&\frac{\phi(x_1,x_2,\ldots,x_{i-1},0,x_{i+1},\ldots,x_n)}{\partial x_n}\notag\\=&\underbrace{\frac{\partial \phi_2(\phi_{n-1}(x_1,x_2,\ldots,x_{i-1},0,x_{i+1},\ldots,x_{n-1}),\sigma(x_n))}{\partial \sigma(x_n)}}_{=0\text{, by the property of $\phi_2$.}}\cdot \frac{\mathrm{d} \sigma(x_n)}{\mathrm{d} x_n}=0.
			\end{align}
			
			Therefore, Eq.~(\ref{equal0}) is valid.
		\end{proof}
  
In the later analysis, we require a lemma concerning the composition of functions in Sobolev spaces: \begin{lemma}[{\cite[Corollary B.5]{guhring2020error}}]\label{composition}
			Let $d,m\in\sN_+$ and $\Omega_1\subset\sR^d$ and $\Omega_2\subset\sR^m$ both be open, bounded, and convex. Then for $\boldsymbol{f}\in W^{1,\infty}(\Omega_1,\sR^m)$ and $g\in W^{1,\infty}(\Omega_2)$ with $ {\rm ran}\boldsymbol{f}\subset \Omega_2$, we have \[\|g\circ\boldsymbol{f}\|_{W^{1,\infty}(\Omega_2)}\le \sqrt{d}{m\max\{\|g\|_{L^{\infty}(\Omega_2)}, \|g\|_{W^{1,\infty}(\Omega_2)}}\|\boldsymbol{f}\|_{W^{1,\infty}(\Omega_1,\sR^m)}\}.\]
		\end{lemma}

  Before presenting the proof, it is essential to introduce three propositions related to ReLU-based DNNs. In particular, for Proposition \ref{point}, it leverages the bit-extraction technique introduced in \cite{bartlett1998almost, bartlett2019nearly} to represent piecewise linear functions on a fixed regular grid with $M$ cells, requiring only $\fO(\sqrt{M})$. (Proposition \ref{point}). \begin{proposition}[{\cite[Proposition 4.3]{lu2021deep}}]\label{step}
			Given any $N,L\in\sN_+$ and $\delta\in\Big(0,\frac{1}{3K}\Big]$ for $K\le N^2L^2$, there exists a $\sigma$-NN $\phi$ with the width $4N+5$ and depth $4L+4$ such that 
			
			\[\phi(x)=k,x\in\left[\frac{k}{K},\frac{k+1}{K}-\delta\cdot 1_{k< K-1}\right], ~k=0,1,\ldots,K-1.\]
		\end{proposition}

		\begin{proposition}[{\cite[Proposition 4.4]{lu2021deep}}]\label{point}
			Given any $N,L,s\in\sN_+$ and $\xi_i\in[0,1]$ for $i=0,1,\ldots N^2L^2-1$, there exists a $\sigma$-NN $\phi$ with the width $16s(N+1)\log_2(8N)$ and depth $(5L+2)\log_2(4L)$ such that
			
			1. $|\phi(i)-\xi_i|\le N^{-2s}L^{-2s}$ for $i=0,1,\ldots N^{2}L^{2}-1$.
			
			2. $0\le \phi(x)\le 1$, $x\in\sR$.
		\end{proposition}

  \begin{proposition}
   [{\cite[Proposition 1]{siegel2022optimal}}]\label{sum}
			Given a sequence of the neural network $\{p_i\}_{i=1}^{M}$, and each $p_i$ is a $\sigma$-NN from $\sR^d\to\sR$ with the width $N$ and depth $L_i$, then $\sum_{i=1}^Mp_i$ is a $\sigma$-NN with the width $N+2d+2$ and  depth $\sum_{i=1}^ML_i$.
		\end{proposition}

  \section{Approximation in Korobov Spaces with Rates in Continuous Function Approximators}\label{continuouss}
In this section, we aim to establish the approximation of DNNs with an optimal rate in continuous function approximation theory. Our approximation error is dependent not only on the width $N$ but also on the depth $L$ of the DNNs. The result, measured by $H^1$ norms, is presented as follows, and the result measured by $L_p$ norm is provided in Corollary \ref{L2}.

  \begin{theorem}\label{H1}
      For any $N, L\in\sN_+$ and $f(\boldsymbol{x})\in X^{2,\infty}(\Omega)$, there exists a $\sigma$-NN $\phi(\boldsymbol{x})$ with the width $C_1N(\log_2 N)^d$ and a depth of $C_2 L(\log_2 L)^d$ such that 
			\begin{align}
      \|f(\boldsymbol{x})-\phi(\boldsymbol{x})\|_{H^1(\Omega)}\le \frac{C_3}{NL}
  \end{align} with $\phi(\boldsymbol{x})|_{\partial \Omega}=0$, where $C_1,C_2$ and $C_3$ are independent with $N$ and $L$, and polynomially dependent on the dimension $d$.$\footnote{In fact, $C_1,C_2$ and $C_3$ can be expressed by $d$ with an explicit formula as we note in the proof of this theorem.  However, the formulas may be very complicated.}$
  \end{theorem}

  Before the proof, we need to approximate the grid functions in the first. \begin{proposition}\label{grid}
      For any $N, L\in\sN_+$ with $|\boldsymbol{l}|_1\le n+d-1,\boldsymbol{1} \leq \boldsymbol{i} \leq 2^{\boldsymbol{l}}-\boldsymbol{1}$, there exists a $\sigma$-NN $\hat{\phi}_{\boldsymbol{l},\boldsymbol{i}}(\boldsymbol{x})$ with the width $9(N+1)+4d-1$ and depth $14 d(d-1) L+1$ such that 
			 \[\|\hat{\phi}_{\boldsymbol{l},\boldsymbol{i}}(\boldsymbol{x})-\phi_{\boldsymbol{l},\boldsymbol{i}}(\boldsymbol{x})\|_{W^{1,\infty}(\Omega)}\le 10d^{\frac{5}{2}}(N+1)^{-7 d L} \cdot 2^{|\boldsymbol{l}|_1},\]with $\text{supp}~\hat{\phi}_{\boldsymbol{l},\boldsymbol{i}}(\boldsymbol{x})\subset\text{supp}~\phi_{\boldsymbol{l},\boldsymbol{i}}(\boldsymbol{x})$.
  \end{proposition} \begin{proof}
     For each hat function $\phi_{l_j, i_j}(x_j)$, it can be expressed as:
\[\phi_{l_j, i_j}(x_j) = \sigma\left(\frac{x_j - i_j \cdot h_{l_j}}{h_{l_j}} - 1\right) - 2\sigma\left(\frac{x_j - i_j \cdot h_{l_j}}{h_{l_j}}\right) + \sigma\left(\frac{x_j - i_j \cdot h_{l_j}}{h_{l_j}} + 1\right).\]
According to Proposition \ref{prop1}, there exists a $\sigma$-NN $\phi_\text{prod}$ with a width of $9(N+1) + d - 1$ and depth of $14 d(d-1) L$ such that $\|\phi_\text{prod}\|_{\mathcal{W}^{1, \infty}([0,1]^d)} \leq 18$ and
\[\left\|\phi_\text{prod}(\boldsymbol{x}) - y_1 y_2 \cdots y_d\right\|_{\mathcal{W}^{1, \infty}([0,1]^d)} \leq 10(d-1)(N+1)^{-7 d L}.\]
Hence, we define
\[\hat{\phi}_{\boldsymbol{l},\boldsymbol{i}}(\boldsymbol{x}) = \phi_\text{prod}(\phi_{l_1, i_1}(x_1),\cdots,\phi_{l_d, i_d}(x_d)),\]
where $\hat{\phi}_{\boldsymbol{l},\boldsymbol{i}}(\boldsymbol{x})$ is a $\sigma$-NN with a width of $9(N+1) + 4d - 1$ and depth of $14 d(d-1) L + 1$. Furthermore, considering Proposition \ref{prop1} and Lemma \ref{composition}, we have:
\[\|\hat{\phi}_{\boldsymbol{l},\boldsymbol{i}}(\boldsymbol{x}) - \phi_{\boldsymbol{l},\boldsymbol{i}}(\boldsymbol{x})\|_{\mathcal{W}^{1,\infty}(\Omega)} = \|(\phi_\text{prod} - y_1 y_2 \cdots y_d)\circ(\phi_{l_1, i_1}(x_1),\cdots,\phi_{l_d, i_d}(x_d))\|_{\mathcal{W}^{1,\infty}(\Omega)}.\]
This leads to:
\[\|\hat{\phi}_{\boldsymbol{l},\boldsymbol{i}}(\boldsymbol{x}) - \phi_{\boldsymbol{l},\boldsymbol{i}}(\boldsymbol{x})\|_{\mathcal{W}^{1,\infty}(\Omega)}\le 10d^{\frac{5}{2}}(N+1)^{-7 d L} \cdot 2^{|\boldsymbol{l}|_1}.\]
Furthermore, if \(\phi_{\boldsymbol{l},\boldsymbol{i}}(\boldsymbol{x})=0\), there exists \(\phi_{l_j, i_j}(x_j) = 0\). As per Proposition \ref{prop1}, we conclude \(\hat{\phi}_{\boldsymbol{l},\boldsymbol{i}}(\boldsymbol{x})= 0\).
  \end{proof}

  \begin{proof}[{Proof of Theorem \ref{H1}}]
      
   Denote \[\phi(\boldsymbol{x})=\sum_{|\boldsymbol{l}|_1 \leq n+d-1} \sum_{\boldsymbol{i} \in \boldsymbol{i}_{\boldsymbol{l}}} v_{\boldsymbol{l}, \boldsymbol{i}} \hat{\phi}_{\boldsymbol{l}, \boldsymbol{i}}(\boldsymbol{x})\] which can be interpreted as a $\sigma$-NN with a width of $\fO(2^nn^{d-1}N)$ and depth of $\fO(L)$, with the error given by\begin{align}
      \|f-\phi\|_{H^1(\Omega)}\le C \left[M^{-1}\left|\log _2 M\right|^{(d-1)}+\sum_{|\boldsymbol{l}|_1 \leq n+d-1}\left\| \sum_{\boldsymbol{i} \in \boldsymbol{i}_{\boldsymbol{l}}}(v_{\boldsymbol{l},\boldsymbol{i}})(\hat{\phi}_{\boldsymbol{l},\boldsymbol{i}}(\boldsymbol{x})-\phi_{\boldsymbol{l},\boldsymbol{i}}(\boldsymbol{x}))\right\|_{H^1(\Omega)}\right],
  \end{align}where the constant $C$ is polynomially dependent on the dimension $d$.$\footnote{In this paper, we consistently employ the symbol $C$ as a constant independent of $M$, $N$, and $L$, which may vary from line to line.}$
  
  Due to $\text{supp}~\hat{\phi}_{\boldsymbol{l},\boldsymbol{i}}(\boldsymbol{x}) \subset \text{supp}~\phi_{\boldsymbol{l},\boldsymbol{i}}(\boldsymbol{x})$, Proposition \ref{grid}, and the fact that a given $\boldsymbol{x}\in\Omega$ belongs to the support of at most one $\phi_{\boldsymbol{l},\boldsymbol{i}}(\boldsymbol{x})$ because they have disjoint supports, we have \begin{align}
      \left\| \sum_{\boldsymbol{i} \in \boldsymbol{i}_{\boldsymbol{l}}}(v_{\boldsymbol{l},\boldsymbol{i}})(\hat{\phi}_{\boldsymbol{l},\boldsymbol{i}}(\boldsymbol{x})-\phi_{\boldsymbol{l},\boldsymbol{i}}(\boldsymbol{x}))\right\|_{H^1(\Omega)}\le 2^{-d}2^{-|\boldsymbol{l}|_1}|f|_{2,\infty}10d^{\frac{5}{2}}(N+1)^{-7 d L}.
  \end{align} Since $2^{-d}\sum_{|\boldsymbol{l}|_1 \leq n+d-1}2^{-|\boldsymbol{l}|_1}<\frac{1}{2\ln 2}\le 1$, we have that \begin{align}
      \sum_{|\boldsymbol{l}|_1 \leq n+d-1}\left\| \sum_{\boldsymbol{i} \in \boldsymbol{i}_{\boldsymbol{l}}}(v_{\boldsymbol{l},\boldsymbol{i}})(\hat{\phi}_{\boldsymbol{l},\boldsymbol{i}}(\boldsymbol{x})-\phi_{\boldsymbol{l},\boldsymbol{i}}(\boldsymbol{x}))\right\|_{H^1(\Omega)}\le C (N+1)^{-7 d L}
  \end{align}where the constant $C$ is polynomially dependent on the dimension $d$. Setting $n=\lfloor\log_2\widetilde{N}\rfloor+\lfloor\log_2\widetilde{L}\rfloor$, the neural network $\phi(\boldsymbol{x})$ can be viewed as a $\sigma$-NN with a depth of $\fO(L)$ a width of $\fO(\widetilde{N}\widetilde{L}(\log_2 \widetilde{N}\log_2 \widetilde{L})^{d-1}N)$ and a width of $\fO(\widetilde{N}\widetilde{L}(\log_2 \widetilde{N}\log_2 \widetilde{L})^{d-1}N)$. It can also be regarded as the sum of a number of $\fO(\widetilde{L}(\log_2 \widetilde{L})^{d-1}/\log_2 \widetilde{N})$ neural networks, each with a width of $\fO(\widetilde{N}(\log_2 \widetilde{N})^dN)$ and a depth of $\fO(L)$. Due to Proposition \ref{sum}, we know that $\phi(\boldsymbol{x})$ is a $\sigma$-NN with a width of $\fO(\widetilde{N}(\log_2 \widetilde{N})^dN)$ and a depth of $\fO(L\widetilde{L}(\log_2 \widetilde{L})^{d-1}/\log_2 \widetilde{N})$.

Setting $N=1$ and $L=\lfloor\log_2\widetilde{N}\rfloor+\lfloor\log_2\widetilde{L}\rfloor$, we have that $\phi(\boldsymbol{x})$ is a $\sigma$-NN with a width of $\fO(\widetilde{N}(\log_2 \widetilde{N})^d)$ and a depth of $\fO(\widetilde{L}(\log_2 \widetilde{L})^d)$.  Furthermore,\begin{align}
      \sum_{|\boldsymbol{l}|_1 \leq n+d-1}\left\| \sum_{\boldsymbol{i} \in \boldsymbol{i}_{\boldsymbol{l}}}(v_{\boldsymbol{l},\boldsymbol{i}})(\hat{\phi}_{\boldsymbol{l},\boldsymbol{i}}(\boldsymbol{x})-\phi_{\boldsymbol{l},\boldsymbol{i}}(\boldsymbol{x}))\right\|_{H^1(\Omega)}\le C (N+1)^{-7 d L}\le C2^{-7d\log_2(\widetilde{N}\widetilde{L})}\le  \frac{C}{\widetilde{N}\widetilde{L}}
  \end{align} 

  Finally, due to \[M\le C{\widetilde{N}\widetilde{L}}(\log_2{\widetilde{N}\widetilde{L}})^{d-1}\]  we obtain that \begin{align}
      \|f-\phi\|_{H^1(\Omega)}\le C \left[M^{-1}\left|\log _2 M\right|^{(d-1)}+\frac{1}{\widetilde{N}\widetilde{L}}\right]\le \frac{C}{\widetilde{N}\widetilde{L}},
  \end{align}where the constant $C$ is polynomially dependent on the dimension $d$. The boundary condition can be directly obtained from $\text{supp}~\hat{\phi}_{\boldsymbol{l},\boldsymbol{i}}(\boldsymbol{x}) \subset \text{supp}~\phi_{\boldsymbol{l},\boldsymbol{i}}(\boldsymbol{x})$.
  \end{proof}

  Following the same idea in the proof, we derive the following Corollary, which describes the approximation of Korobov spaces by deep neural networks measured by $L^2$ norms:
  \begin{corollary}\label{L2}
      For any $N, L\in\sN_+$ and $f(\boldsymbol{x})\in X^{2,\infty}(\Omega)$, there exists a $\sigma$-NN $\phi(\boldsymbol{x})$ with the width $C_1N(\log_2 N)^{3d}$ and a depth of $C_2 L(\log_2 L)^{3d}$ such that 
			\begin{align}
      \|f(\boldsymbol{x})-\phi(\boldsymbol{x})\|_{L_p(\Omega)}\le \frac{C_3}{N^2L^2}
  \end{align} with $1\le p\le \infty$ and $\phi(\boldsymbol{x})|_{\partial \Omega}=0$, where $C_1,C_2$ and $C_3$ are independent with $N$ and $L$, and polynomially dependent on the dimension $d$.
  \end{corollary}

  Note that the number of parameters is $\fO(N^2L(\log_2 L)^{3d}(\log_2 N)^{6d})$, with an error of $\fO(N^{-2}L^{-2})$. This result is consistent with the findings in \cite{montanelli2019new} when we fix $N$ and consider depth $L$. Our result achieves the optimal approximation rate for continuous function approximation, as established in \cite{devore1989optimal}. The main improvement in our findings, compared to \cite{montanelli2019new, mao2022approximation, blanchard2021shallow, suzuki2018adaptivity}, lies in our consideration of depth flexibility in DNNs and the establishment of the approximation rate measured by the $H^1$ norms.

  \section{Super Convergence Rates for Korobov Functions in $L_p$-norm}\label{LopKK}
  In this section, our primary objective is to establish Deep Neural Networks (DNNs) as function approximators within Korobov Spaces with a \textit{super-convergence} rate, surpassing existing works. In this context, we focus solely on error measurement using $L_p$, where $1\le p\le \infty$. In the next section, we will extend our error analysis to include Sobolev norms, specifically the $H^1$ norm.

  \begin{theorem}\label{L2 K}
    For any $f\in X^{2,\infty}(\Omega)$ and $|f|_{2,\infty}\le 1$, $N, L\in\sN_+$, there is a $\sigma$-NN $k(\boldsymbol{x})$ with $3^d(64d(N+3)(\log_2(8N))^{d+1})$  width and $(33L+2)(\log_2(4L))^{d+1}+2d$ depth such that\begin{align}
        \left\|k(\boldsymbol{x})-f(\boldsymbol{x})\right\|_{L_p(\Omega)}\le  CN^{-4}L^{-4}(\log_2N)^{d-1}(\log_2L)^{d-1},
   \end{align} where $C$ is the constant independent with $N,L$, and polynomially dependent on the dimension $d$, $1\le p\le \infty$.
\end{theorem}

  We initially employ Deep Neural Networks (DNNs) to approximate functions \(f \in X^{2,\infty}(\Omega)\) across the entire domain \(\Omega\), excluding a small set. The domain is precisely defined as follows:
\begin{definition} \label{delta}For any $n\ge 1$, let $\delta=\frac{1}{2^{n+2}}$. For any $|\boldsymbol{l}|_1\le {n+d-1}$
      \begin{align}
			    \Omega_{\boldsymbol{l},\delta}&=\bigcup_{\boldsymbol{i}\in \boldsymbol{i}_{\boldsymbol{l}}}\Omega_{\boldsymbol{l},\boldsymbol{i},\delta},\\\notag
       \Omega_{\boldsymbol{l},\boldsymbol{i},\delta}&= \prod_{r=1}^d\left[\frac{2i_r-1}{2K_r},\frac{2i_r+1}{2K_r}-\delta\cdot 1_{k< K_r-1}\right],~K_r=2^{l_r-1}.
			\end{align} Then we define \[\Omega_{\delta}=\bigcap_{|\boldsymbol{l}|_1\le n+d-1}\Omega_{\boldsymbol{l},\delta}.\]
  \end{definition}

  \begin{proposition}\label{main1}
			For any $f\in X^{2,\infty}(\Omega)$ with $p\ge1$ and $|f|_{2,\infty}\le 1$, $N, L\in\sN_+$, there is a $\sigma$-NN $\tilde{k}(\boldsymbol{x})$ with $64d(N+2)(\log_2(8N))^{d+1}$  width and $(33L+2)(\log_2(4L))^{d+1}$ depth, such that\begin{align}
        \left\|\tilde{k}(\boldsymbol{x})-f(\boldsymbol{x})\right\|_{L_\infty(\Omega_\delta)}\le  CN^{-4}L^{-4}(\log_2N)^{d-1}(\log_2L)^{d-1}.
   \end{align} where $C$ is the constant independent with $N,L$, and polynomially dependent on the dimension $d$.
		\end{proposition}

  \begin{proof}

  In the proof, our first step involves utilizing Deep Neural Networks (DNNs) to approximate  $ \sum_{\boldsymbol{i} \in \boldsymbol{i}_{\boldsymbol{l}_*}} v_{\boldsymbol{l}_*, \boldsymbol{i}} \phi_{\boldsymbol{l}_*, \boldsymbol{i}}(\boldsymbol{x})$ for $|\boldsymbol{l}_*|_1\le n+d-1$.
It's noteworthy that \( \sum_{\boldsymbol{i} \in \boldsymbol{i}_{\boldsymbol{l}_*}} v_{\boldsymbol{l}_*, \boldsymbol{i}} \phi_{\boldsymbol{l}_*, \boldsymbol{i}}(\boldsymbol{x})\) can be reformulated as \[g_{\boldsymbol{l}_*}(\boldsymbol{x})=p_{\boldsymbol{l}_*}(\boldsymbol{x})q_{\boldsymbol{l}_*}(\boldsymbol{x}).\] Here, \(q(\boldsymbol{x})\) is a piecewise function on \([0,1]^d\) defined by
\begin{equation}
    q_{\boldsymbol{l}_*}(\boldsymbol{x})=v_{\boldsymbol{l}_*,\boldsymbol{i}}, \text{ for } \boldsymbol{x}\in \prod_{k=1}^d\left[(i_k-1)\cdot 2^{-l_k},(i_k+1)\cdot 2^{-l_k}\right].
\end{equation}
Meanwhile, \(p(\boldsymbol{x})\) is a piecewise-polynomial defined as
\begin{align}
    p_{l_k}(x)&=\sum_{s=1}^{2^{k-1}}\phi\left(\frac{x_j-(2s-1) \cdot h_{l_j}}{h_{l_j}}\right) \notag \\
    p_{\boldsymbol{l}_*}(\boldsymbol{x})&=\prod_{k=1}^d p_{l_k}(x_k).
\end{align}
Here, \(p_{l_k}(x)\) represents a deformed sawtooth function.

  Let $N_r^2L_r^2\ge K_r= 2^{l_r-1}$. By leveraging Proposition \ref{step}, we ascertain the existence of a $\sigma$-NN $\phi_r(x)$ with a width of $4N_r+5$ and a depth of $4L_r+4$ such that
			\[\phi_r(x)=k,x\in\left[\frac{k}{K_r},\frac{k+1}{K_r}-\delta\cdot 1_{k< K_r-1}\right], ~k=1,\ldots,K_r-1,\]with $\delta\in\Big(0,\frac{1}{3K^r}\Big]$. 
   Then define \[\boldsymbol{\phi}_2(\boldsymbol{x})=\left[\frac{\phi_1(x_1)}{K_1},\frac{\phi_2(x_2)}{K_2},\ldots,\frac{\phi_d(x_d)}{K_d}\right]^T.\] 
   
   For each $p=0,1,\ldots,\prod_{r=1}^dK_r-1$, there is a bijection\[\boldsymbol{\eta}(p)=[\eta_1,\eta_2,\ldots,\eta_d]\in\prod_{r=1}^d \{0,\ldots,K_r-1\}\] such that $\sum_{j=1}^d\eta_j\prod_{r=1}^{j-1}K_r=p$.
   
  Set $C_{\alpha,\boldsymbol{l}_*}=2^{-d-|\boldsymbol{l}_*|_1}|f|_{2,\infty}\ge |v_{\boldsymbol{l}_*,\boldsymbol{i}}|$ for all $\boldsymbol{i}$, and define 
\begin{equation}
    \xi_{\boldsymbol{\alpha},\boldsymbol{l}_*,\boldsymbol{i}}=\frac{v_{\boldsymbol{l}_*,\boldsymbol{i}}+C_{\alpha,\boldsymbol{l}_*}}{2C_{\alpha,\boldsymbol{l}_*}}\in[0,1].
\end{equation} 

Based on Proposition \ref{point}, there exists a neural network $\tilde{\phi}_{\boldsymbol{\alpha}}(x)$ with a width of $16s(\tilde{N}+1)\log_2(8\tilde{N})$ and a depth of $(5\tilde{L}+2)\log_2(4\tilde{L})$ such that 
\[\left|\tilde{\phi}_{\boldsymbol{\alpha}}\left(\sum_{j=1}^d\frac{i_j-1}{2}\prod_{r=1}^{j-1}K_r\right)- \xi_{\boldsymbol{\alpha},\boldsymbol{l}_*,\boldsymbol{i}}\right|\le \tilde{L}^{-2s}\tilde{N}^{-2s}\]
for $|\boldsymbol{l}_*|_1 \leq n+d-1$ and $\boldsymbol{i} \in \boldsymbol{i}_{\boldsymbol{l}_*}$. Therefore, we define 
\begin{equation}
    \phi_{\boldsymbol{\alpha}}(\boldsymbol{x}):=2C_{\alpha,\boldsymbol{l}_*}\tilde{\phi}_{\boldsymbol{\alpha}}\left(\sum_{j=1}^dx_j\prod_{r=1}^{j}K_r\right)-C_{\alpha,\boldsymbol{l}_*}.
\end{equation} 

Consequently, we find that 
\begin{align}
    | \phi_{\boldsymbol{\alpha}}(\boldsymbol{\phi}_2(\boldsymbol{x}))-q_{\boldsymbol{l}_*}(\boldsymbol{x})|&=\left|2C_{\alpha,\boldsymbol{l}_*}\tilde{\phi}_{\boldsymbol{\alpha}}\left(\sum_{j=1}^d\frac{i_j-1}{2}\prod_{r=1}^{j-1}K_r\right)-C_{\alpha,\boldsymbol{l}_*}-v_{\boldsymbol{l}_*,\boldsymbol{i}}\right|\notag\\
    &\le 2C_{\alpha,\boldsymbol{l}_*} \left|\tilde{\phi}_{\boldsymbol{\alpha}}\left(\sum_{j=1}^d\frac{i_j-1}{2}\prod_{r=1}^{j-1}K_r\right)- \xi_{\boldsymbol{\alpha},\boldsymbol{l}_*,\boldsymbol{i}}\right|\le 2C_{\alpha,\boldsymbol{l}_*}\tilde{L}^{-2s}\tilde{N}^{-2s}
\end{align}
for $\boldsymbol{x}\in \Omega_{\boldsymbol{l}_*,\boldsymbol{i},\delta}=\prod_{r=1}^d\left[\frac{i_r-1}{2K_r},\frac{2i_r+1}{2K_r}-\delta\cdot 1_{k< K_r-1}\right]$.

 Since there are at most $2^{n-1}$ elements in $\boldsymbol{i}_{\boldsymbol{l}_*}$, we set $\tilde{L}^{2}\tilde{N}^{2}\ge 2^{n-1}$. Above all, we can let $L=\max\{\tilde{L},L_r\}=2^{n_1}$ and $N=\max\{\tilde{L},N_r\}=2^{n_2}$, where $2(n_1+n_2)$ is the smallest even number larger or equal to $n-1$. Then $\boldsymbol{\phi}_2(\boldsymbol{x})$ is a $\sigma$-NN with $4dN+5d$ width and $4L+4$ depth. $\phi_{\boldsymbol{\alpha}}(\boldsymbol{x})$ is a $\sigma$-NN with the width $16s(N+1)\log_2(8N)$ and depth $(5L+2)\log_2(4L)$. Above all, $\phi_{\boldsymbol{\alpha}}(\boldsymbol{\phi}_2(\boldsymbol{x}))$ is a $\sigma$-NN with $16sd(N+1)\log_2(8N)$ and depth $(9L+2)\log_2(4L)$. We denote this $\phi_{\boldsymbol{\alpha}}(\boldsymbol{\phi}_2(\boldsymbol{x}))$ as $s_{\boldsymbol{l}_*}(\boldsymbol{x})$, which can approximate $q_{\boldsymbol{l}_*}(\boldsymbol{x})$ well on $\Omega_{\boldsymbol{l}_*,\delta}$. Set $\delta=\frac{1}{2^{n+2}}$, we also find $s_{\boldsymbol{l}_*}(\boldsymbol{x})$, which can approximate $q_{\boldsymbol{l}_*}(\boldsymbol{x})$ well on $\Omega_{\boldsymbol{l}_*,\delta}$ for $|\boldsymbol{l}_*|_1\le n+d-1$.

Recall \[\Omega_{\delta}=\bigcap_{|\boldsymbol{l}_*|_1\le n+d-1}\Omega_{\boldsymbol{l}_*,\delta},\] then $s_{\boldsymbol{l}_*}(\boldsymbol{x})$, which can approximate $q_{\boldsymbol{l}_*}(\boldsymbol{x})$ well on $\Omega_{\delta}$ for $|\boldsymbol{l}_*|_1\le n+d-1$.

 Next, we aim to approximate $p_{\boldsymbol{l}_*}(\boldsymbol{x})=\prod_{k=1}^d p_{l_k}(x_k)$. The proof relies on leveraging the periodicity of each $p_{l_k}(x_k)$. We first define
\[\psi_1(x)=p_{l_k}(x),~x\in[0,2^{-n+1+n_1}], \text{~otherwise is }0.\]

Then, $\psi_1(x)$ is a neural network (NN) with $4N$ width and $1$ depth.
\begin{figure}[h!]
    \centering
    \includegraphics[scale=0.57]{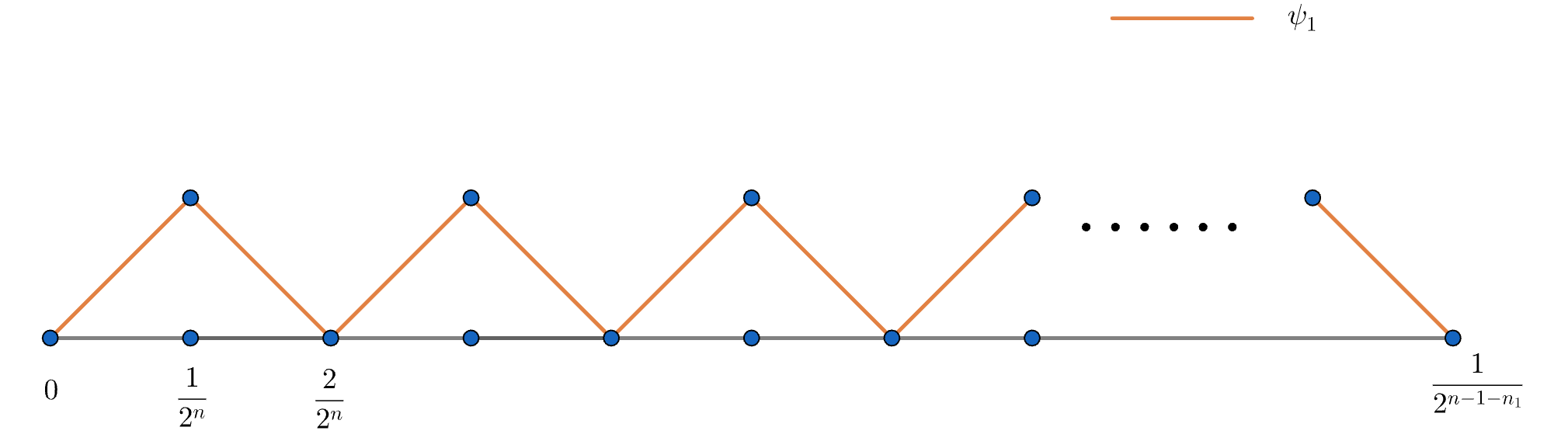}
    \caption{The sawtooth functions $\psi_1$}
    \label{psi4}
\end{figure}

Next, we construct $\psi_i$ for $i=2,3,4$ based on the symmetry and periodicity of $g_i$. $\psi_2$ is the function with period $\frac{2}{NL^2}$ in $\left[0,\frac{1}{L^2}\right]$, and each period is a hat function with gradient 1. $\psi_3$ is the function with period $\frac{2}{L^2}$ in $\left[0,\frac{1}{L}\right]$, and each period is a hat function with gradient 1. $\psi_4$ is the function with period ${\frac{2}{L}}$ in $\left[0,1\right]$, and each period is a hat function with gradient 1. Similar with $\psi_1$, $\psi_2$ is a network with $4N$ width and one layer. Due to Proposition \ref{sum}, we know that $\psi_3$ and $\psi_4$ is a network with $7$ width and $L+1$ depth.

%For the remaining operations, we reflect $\psi_1(x)$ a total of $n_1+2n_2$ times. For each reflection, it composes a hat function. Therefore, $p_{l_k}(x_k)$ is a DNN with $3N$ width and $2\log_2L+\log_2N$ depth. Since $\log_2 N\le L$, $p_{l_k}(x_k)$ is a DNN with $3N$ width and $2\log_2L+L$ depth.

Finally, by Proposition \ref{prop1}, there exists a $\sigma$-NN $w_{\boldsymbol{l}_*}(\boldsymbol{x})$ with $4(N+d+3)+s'-1$ width and $16s'(s'-1)L$ depth, such that
\begin{equation}
    \|w_{\boldsymbol{l}_*}(\boldsymbol{x})-p_{\boldsymbol{l}_*}(\boldsymbol{x})\|_{L_\infty(\Omega)}\le 10(s'-1)(N+1)^{-7 s' L}.
\end{equation}

Based on Proposition \ref{2prop}, since $|v_{\boldsymbol{l}_*,\boldsymbol{i}}|\le 1$, there exists $\hat{\phi}$ with a width of $15N$ and a depth of $24L$ such that
\begin{equation}
    \left\|\hat{\phi}(x,y)-xy\right\|_{L^{\infty}([-1,1]^2)}\le 6N^{-8L}.
\end{equation}
Therefore, we have 
\begin{align}
    &\|\hat{\phi}(s_{\boldsymbol{l}_*}(\boldsymbol{x}),w_{\boldsymbol{l}_*}(\boldsymbol{x}))-p_{\boldsymbol{l}_*}(\boldsymbol{x})q_{\boldsymbol{l}_*}(\boldsymbol{x})\|_{L_\infty(\Omega_\delta)}\notag\\
    \le&\|\hat{\phi}(s_{\boldsymbol{l}_*}(\boldsymbol{x}),w_{\boldsymbol{l}_*}(\boldsymbol{x}))-s_{\boldsymbol{l}_*}(\boldsymbol{x})w_{\boldsymbol{l}_*}(\boldsymbol{x})\|_{L_\infty(\Omega_\delta)}\notag\\
    &+\|s_{\boldsymbol{l}_*}(\boldsymbol{x})w_{\boldsymbol{l}_*}(\boldsymbol{x})-s_{\boldsymbol{l}_*}(\boldsymbol{x})p_{\boldsymbol{l}_*}(\boldsymbol{x})\|_{L_\infty(\Omega_\delta)}\notag\\
    &+\|s_{\boldsymbol{l}_*}(\boldsymbol{x})p_{\boldsymbol{l}_*}(\boldsymbol{x})-p_{\boldsymbol{l}_*}(\boldsymbol{x})q_{\boldsymbol{l}_*}(\boldsymbol{x})\|_{L_\infty(\Omega_\delta)}\notag\\
    \le&6N^{-12L}+20(s'-1)(N+1)^{-7 s' L}+4C_{\alpha,\boldsymbol{l}_*}L^{-2s}N^{-2s}.
\end{align}
Setting $s'=s=2$, we notice that 
\begin{align}
    10(s'-1)(N+1)^{-7 s' L}&=20(N+1)^{-14L}\le 20(N+1)^{-4(L+1)}\le 20 N^{-4}L^{-4}\notag\\
    6N^{-8L}&=6N^{4L}\cdot N^{-4(L+1)}\le  6 N^{-4}L^{-4}.
\end{align}

Above all, we have that there exists a $\sigma$-NN $\psi_{\boldsymbol{l}_*}$ with $64d(N+1)\log_2(8N)$ width and $(33L+2)\log_2(4L)$ depth such that 
\begin{equation}
    \left\|\psi_{\boldsymbol{l}_*}(\boldsymbol{x})-\sum_{\boldsymbol{i} \in \boldsymbol{i}_{\boldsymbol{l}_*}} v_{\boldsymbol{l}_*, \boldsymbol{i}} \phi_{\boldsymbol{l}_*, \boldsymbol{i}}(\boldsymbol{x})\right\|_{L_\infty(\Omega_\delta)}\le (26+4C_{\alpha,\boldsymbol{l}_*})N^{-4}L^{-4}.
\end{equation}

 Similarly, we can find $\sigma$-NNs $\{\psi_{\boldsymbol{l}}(\boldsymbol{x})\}_{|\boldsymbol{l}|_1\le n+d-1}$ for other $\sum_{\boldsymbol{i} \in \boldsymbol{i}_{\boldsymbol{l}}} v_{\boldsymbol{l}, \boldsymbol{i}} \phi_{\boldsymbol{l}, \boldsymbol{i}}(\boldsymbol{x})$ for other $\boldsymbol{l}$. Since there are at most $n^d=(2\log_2(NL)+1)^d$ satisfied $|\boldsymbol{l}|_1\le n+d-1$, we can have a $\sigma$-NN $\tilde{k}(\boldsymbol{x})$ with 
\begin{equation}
    \text{width } 32d(N+1)\log_2(8N)(2\log_2(NL)+1)^d
\end{equation}
and 
\begin{equation}
    \text{depth } (33L+2)\log_2(4L)
\end{equation}
such that 
\begin{equation}
    \left\|\tilde{k}(\boldsymbol{x})-f_1^{(n)}(\boldsymbol{x})\right\|_{L_\infty(\Omega_\delta)}\le (52+8C_{\alpha,\boldsymbol{l}_*})N^{-4}L^{-4}.
\end{equation}

Thanks to Proposition \ref{sum}, $\tilde{k}(\boldsymbol{x})$ can be expressed as 
\begin{equation}
    \frac{32d(N+1)\log_2(8N)(2\log_2(NL)+1)^d}{(\log_2L)^d} \le 64d(N+2)(\log_2(8N))^{d+1}
\end{equation}
width and 
\begin{equation}
    (33L+2)(\log_2(4L))^{d+1}
\end{equation}
depth.

Finally, combining with Lemma \ref{err-first} and $n=2\log_2(NL)+1$, we have that 
\begin{align}
    \left\|\tilde{k}(\boldsymbol{x})-f(\boldsymbol{x})\right\|_{L_\infty(\Omega_\delta)}\le &(52+8C_{\alpha,\boldsymbol{l}_*})N^{-4}L^{-4}+CN^{-4}L^{-4}\frac{(2\log_2(NL)+1)^{3(d-1)}}{(2\log_2(NL)+1)^{2(d-1)}}\notag\\
    \le & CN^{-4}L^{-4}(\log_2N)^{d-1}(\log_2L)^{d-1}.
\end{align}
   
\end{proof}

The next step is to approximate $f\in X^{2,\infty}(\Omega)$ over the entire domain. To begin, we need to establish that $f\in X^{2,\infty}(\Omega)$ is a continuous function.

\begin{proposition}\label{continuous}
    Suppose that $f\in X^{2,\infty}(\Omega)$, then $f$ is a continuous function on $\Omega$.
\end{proposition}
\begin{proof}
   Let $S_N(\boldsymbol{x})$ be the $N$th partial sum of the Fourier series of $f$, i.e.,
\[S_N(\boldsymbol{x})=\frac{1}{(2\pi)^{\frac{d}{2}}}\sum_{|\boldsymbol{k}|_2\le N }\hat{f}_{\boldsymbol{k}}e^{\mathrm{i} \boldsymbol{k}\boldsymbol{x}}.\]

Note that
\begin{equation}
    \left|\frac{\partial^{2d} f(\boldsymbol{x})}{\partial x_1^2\cdots\partial x_d^2}\right|=\frac{1}{(2\pi)^{\frac{d}{2}}}\sum_{\boldsymbol{k}\in \sN^d }(k_1^2k_2^2\cdots k_d^2)\hat{f}_{\boldsymbol{k}}e^{\mathrm{i} \boldsymbol{k}\boldsymbol{x}}.
\end{equation}

Since $\left\|\frac{\partial^{2d} f(\boldsymbol{x})}{\partial x_1^2\cdots\partial x_d^2}\right\|_{L_\infty(\Omega)}$ exists, then
\(\left\|\frac{\partial^{2d} f(\boldsymbol{x})}{\partial x_1^2\cdots\partial x_d^2}\right\|_{L_2(\Omega)}\)
exists. Therefore, we know that
\[C_f:=\sum_{\boldsymbol{k}\in \sN^d }(k_1^2k_2^2\cdots k_d^2)^2|\hat{f}_{\boldsymbol{k}}|^2\]
exists.

For any $\boldsymbol{x}\in\Omega$, we have
\begin{equation}
    \begin{aligned}
        \left|S_N(\boldsymbol{x})-f(\boldsymbol{x})\right| & \leq \frac{1}{(2\pi)^{\frac{d}{2}}} \sum_{N<|\boldsymbol{k}|_2 }\left|\hat{f}_{\boldsymbol{k}}\right| \\
        & =\frac{1}{(2\pi)^{\frac{d}{2}}} \sum_{N<|\boldsymbol{k}|_2 }(k_1^2k_2^2\cdots k_d^2)\left|\hat{f}_{\boldsymbol{k}}\right| \frac{1}{(k_1^2k_2^2\cdots k_d^2)} \\
        & \leq \frac{1}{(2\pi)^{\frac{d}{2}}}\left[\sum_{N<|\boldsymbol{k}|_2 }(k_1^2k_2^2\cdots k_d^2)^2\left|\hat{f}_{\boldsymbol{k}}\right|^2\right]^{1 / 2}\left[\sum_{N<|\boldsymbol{k}|_2} \frac{1}{(k_1^2k_2^2\cdots k_d^2)^{2 }}\right]^{1 / 2} \\
        & \leq C\left(\int_N^\infty r^{-3d-1}\mathrm{d} r\right)^{\frac{1}{2}}\le \frac{C}{N^{\frac{3}{2}d}}.
    \end{aligned}
\end{equation}

Therefore, $\lim_{N\to\infty}S_N=f$ uniformly converges. Hence, $f$ is continuous.
\end{proof}

Next, the following lemma establishes a connection between the approximation on $\Omega_\delta$ and that in the whole domain.

\begin{lemma}[{\cite{lu2021deep,shen2022optimal}}]\label{link}
 Given any $\varepsilon>0, N, L, K \in \sN^{+}$, and $\delta \in\left(0, \frac{1}{3 K}\right]$, assume $f$ is a continuous function in $C\left([0,1]^d\right)$ and $\widetilde{\phi}$ can be implemented by a ReLU network with width $N$ and depth $L$. If
$$
|f(\boldsymbol{x})-\widetilde{\phi}(\boldsymbol{x})| \leq \varepsilon, \quad \text { for any } \boldsymbol{x} \in \Omega_\delta,
$$
then there exists a function $\phi$ implemented by a new ReLU network with width $3^d(N+4)$ and depth $L+2 d$ such that
$$
|f(\boldsymbol{x})-\phi(\boldsymbol{x})| \leq \varepsilon+d \cdot \omega_f(\delta), \quad \text { for any } \boldsymbol{x} \in[0,1]^d \text {, }
$$ where \begin{equation}
\omega_f(r):=\sup \left\{|f(\boldsymbol{x})-f(\boldsymbol{y})|: \boldsymbol{x}, \boldsymbol{y} \in[0,1]^d,\|\boldsymbol{x}-\boldsymbol{y}\|_2 \leq r\right\}, \quad \text { for any } r \geq 0.
\end{equation}
\end{lemma}

Now, leveraging Propositions \ref{main1} and \ref{continuous} along with Lemma \ref{link}, we can derive the approximation of Korobov spaces with a super-convergence rate.

\begin{proof}[Proof of Theorem \ref{L2 K}]
    Based on Propositions \ref{main1} and \ref{continuous}, along with Lemma \ref{link}, for given $N$, $L$, and $d$, we set $\delta$ to be sufficiently small to ensure
\[
d \cdot \omega_f(\delta) \le N^{-4}L^{-4}(\log_2N)^{d-1}(\log_2L)^{d-1}.
\]

Then, there exists a $\sigma$-NN $k(\boldsymbol{x})$ with $3^d(64d(N+3)(\log_2(8N))^{d+1})$ width and $2d+(33L+2)(\log_2(4L))^{d+1}$ depth such that
\begin{align}
    \left\|k(\boldsymbol{x}) - f(\boldsymbol{x})\right\|_{L_\infty(\Omega)} \le CN^{-4}L^{-4}(\log_2N)^{d-1}(\log_2L)^{d-1},
\end{align}
where $C$ is a constant independent of $N$ and $L$, and polynomially dependent on the dimension $d$. Furthermore, we have
\begin{align}
    \left\|k(\boldsymbol{x}) - f(\boldsymbol{x})\right\|_{L_p(\Omega)} \le CN^{-4}L^{-4}(\log_2N)^{d-1}(\log_2L)^{d-1},
\end{align}
for any $p \in [1, \infty]$.
\end{proof}

The approximation rate in Theorem \ref{L2} is significantly superior to that in Corollary \ref{L2}. This error outperforms the results in \cite{mao2022approximation, montanelli2019new,blanchard2021shallow}. Furthermore, our result is nearly optimal based on the following theorem in the $X^{2,\infty}$ case. \begin{theorem}\label{Optimality}
			Given any $\rho, C_{1}, C_{2}, C_{3}, J_0>0$ and $n,d\in\sN^+$, there exist $N,L\in\sN$ with $NL\ge J_0$ and $f\in X^{2,\infty}$ with $|f|_{2,\infty}\le 1$, such that\begin{equation}
				\inf_{\phi\in\fK}\|\phi-f\|_{L^{\infty}(\Omega)}> C_{3}L^{-4-\rho}N^{-4-\rho},
			\end{equation} where \[\fK:=\{\text{$\sigma$-NNs in $\sR^d$ with the width $C_{1} N(\log_2 N)^{d+1}$ and depth $C_{2} L(\log_2 L)^{d+1}$}\}.\]
		\end{theorem}

  In order to prove this, we need a definition called Vapnik--Chervonenkis dimension (VC-dimension):
  \begin{definition}[VC-dimension \cite{abu1989vapnik}]
		Let $H$ denote a class of functions from $\fX$ to $\{0,1\}$. For any non-negative integer $m$, define the growth function of $H$ as \[\Pi_H(m):=\max_{x_1,x_2,\ldots,x_m\in \fX}\left|\{\left(h(x_1),h(x_2),\ldots,h(x_m)\right): h\in H \}\right|.\] The Vapnik--Chervonenkis dimension (VC-dimension) of $H$, denoted by $\text{VCdim}(H)$, is the largest $m$ such that $\Pi_H(m)=2^m$. For a class $\fG$ of real-valued functions, define $\text{VCdim}(\fG):=\text{VCdim}(\text{sgn}(\fG))$, where $\text{sgn}(\fG):=\{\text{sgn}(f):f\in\fG\}$ and $\text{sgn}(x)=1[x>0]$.
		
	\end{definition}

 \begin{lemma}[\cite{bartlett2019nearly}]\label{vcdim}
		For any $N,L,d\in\sN_+$, there exists a constant $\bar{C}$ independent with $N,L$ such that	\begin{equation}
		\text{VCdim}(\Phi)\le \bar{C} N^2L^2\log_2 L\log_2 N,\end{equation} $ \Phi:=\left\{\phi:\phi\text{ is a $\sigma$-NN in $\sR^d$ with width$\le N$ and depth$\le L$}\right\}$.
	\end{lemma}

 \begin{lemma}[\cite{siegel2022optimal}]\label{popt}
     Let $ \Omega=[0,1]^d$ and suppose that $K$ is a translation invariant class of functions whose VC-dimension is at most $n$. By translation invariant we mean that $f \in K$ implies that $f(\cdot-v) \in K$ for any fixed vector $v \in \mathbb{R}^d$. Then there exists an $f \in W^{s,\infty}(\Omega)$ such that
$$
\inf _{g \in K}\|f-g\|_{L_p(\Omega)} \geq C(d, p) n^{-\frac{s}{d} }\|f\|_{ W^{s,\infty}(\Omega)} .
$$
 \end{lemma}

  \begin{proof}[Proof of Theorem \ref{Optimality}]
      Define \[\tilde{\mathcal{K}}:=\{\text{$\sigma$-NNs in } \mathbb{R} \text{ with the width } C_{1} N(\log_2 N)^{d+1} \text{ and depth } C_{2} L(\log_2 L)^{d+1}\}.\] Due to \cite{bartlett2019nearly}, we know that
\[\text{VCdim}(\tilde{\mathcal{K}})= C N^2L^2(\log_2 L)^{2d+3}(\log_2 N)^{2d+3}.\]
Based on Lemma \ref{popt}, there exists a $\tilde{f} \in W^{2,\infty}([0,1])$ with $\|\tilde{f}\|_{W^{2,\infty}}\le 1$ such that
\[
\inf _{g \in K}\|\tilde{f}-g\|_{L_p([0,1])} \geq C(d, p) n^{-\frac{2}{d} }\|\tilde{f}\|_{ W^{2,\infty}([0,1])}.
\]
Now we can define $f(\mathbf{x})=\tilde{f}(x_1)$ which belongs to $X^{2,\infty}$. Then we know that for any $\rho>0$, there is an $f\in W^{2d,\infty}(\Omega)\subset X^{2,\infty}(\Omega)$ with $|f|_{2,\infty}\le 1$ and $C_3>0$ such that 
\begin{align}
    &\inf _{\phi \in \mathcal{K}}\|f-\phi\|_{L_\infty(\Omega)} \geq \inf _{\phi \in \mathcal{K}}\|\tilde{f}(x_1)-\phi(x_1,\ldots,x_d)\|_{L_\infty(\Omega)}\notag\\
    &\ge \inf _{\phi \in \tilde{\mathcal{K}}}\|\tilde{f}(x_1)-\phi(x_1)\|_{L_\infty([0,1])}\notag\\
    &\ge C(d, p) N^{-4}L^{-4}(\log_2 L)^{-4d-6}(\log_2 N)^{-4d-6} > C_{3}L^{-4-\rho}N^{-4-\rho}.
\end{align}
The second inequality is due to the fact that for any fixed $x_2,x_3,\ldots,x_d$, $\phi(x_1,\ldots,x_d)$ belongs to $\tilde{\mathcal{K}}$ with respect to $x_1$.
  \end{proof}

  \section{Super Convergence Rates for Korobov Functions in the $H^1$ norm}\label{hop}

  In this section, we will extend our analysis from Section \ref{LopKK} to the case of $H^1$ norms. This extension ensures that our DNNs can approximate functions in Korobov spaces with minimal discrepancies in both magnitude and derivative, achieving optimality and demonstrating the \textit{super-convergence} rate.

  \begin{theorem}\label{H1 K}
			For any $f\in X^{2,\infty}(\Omega)$ and $|f|_{2,\infty}\le 1$, $\|f\|_{W^{1,\infty}(\Omega)}\le 1$, $N, L\in\sN_+$, there is a $\sigma$-NN $k(\boldsymbol{x})$ with $2^{d+6}d(N+2)(\log_2(8N))^{d+1}$  width and $(47L+2)(\log_2(4L))^{d+1}$ depth such that\[\|f(\boldsymbol{x})-k(\boldsymbol{x})\|_{H^{1}([0,1]^d)}\le C{ N^{-2} L^{-2}}(\log_2N)^{d-1}(\log_2L)^{d-1},\]where $C$ is the constant independent with $N,L$.
		\end{theorem}

  First of all, define a sequence of subsets of $\Omega$:
		\begin{definition}\label{omega}Given $K,d\in\sN^+$, and for any $\boldsymbol{m}=(m_1,m_2,\ldots,m_d)\in\{1,2\}^d$, we define $
				\Omega_{\boldsymbol{m}}:=\prod_{j=1}^d\Omega_{m_j},
		$  where $
			\Omega_1:=\bigcup_{i=0}^{K-1}\left[\frac{i}{K},\frac{i}{K}+\frac{3}{4K}\right],~\Omega_2:=\bigcup_{i=0}^{K}\left[\frac{i}{K}-\frac{1}{2K},\frac{i}{K}+\frac{1}{4K}\right]\cap [0,1]$.\end{definition}

  Note that $\Omega_\textbf{1}=\Omega_\delta$ when $K=\frac{1}{2^n}$, where $n$ and $\Omega_\delta$ are defined in Definition \ref{delta}.

        Then we define a partition of unity $\{g_{\boldsymbol{m}}\}_{\boldsymbol{m}\in\{1,2\}^d}$ on $[0,1]^d$ with $\text{supp }g_{\boldsymbol{m}}\cap[0,1]^d\subset \Omega_{\boldsymbol{m}}$ for each $\boldsymbol{m}\in\{1,2\}^d$: \begin{definition}\label{gm}
			Given $K,d\in\sN_+$, we define\begin{align}
				g_1(x):=
				\begin{cases}
					1,~&x\in \left[\frac{i}{K}+\frac{1}{4K},\frac{i}{K}+\frac{1}{2K}\right] \\
					0,~&x\in\left[\frac{i}{K}+\frac{3}{4K},\frac{i+1}{K}\right] \\
					4K\left(x-\frac{i}{K}\right),~&x\in\left[\frac{i}{K},\frac{i}{K}+\frac{1}{4K}\right]\\
					-4K\left(x-\frac{i}{K}-\frac{3}{4K}\right),~&x\in\left[\frac{i}{K}+\frac{1}{2K},\frac{i}{K}+\frac{3}{4K}\right]
				\end{cases},~g_2(x):=g_1\left(x+\frac{1}{2K}\right),
			\end{align}for $i\in\sZ$. For any $\boldsymbol{m}=(m_1,m_2,\ldots,m_d)\in\{1,2\}^d$, define $
			g_{\boldsymbol{m}}(\boldsymbol{x})=\prod_{j=1}^d g_{m_j}(x_j),~\boldsymbol{x}=(x_1,x_2,\ldots,x_d)$.\end{definition}
		
		\begin{figure}[h!]
			\centering
			\includegraphics[scale=0.47]{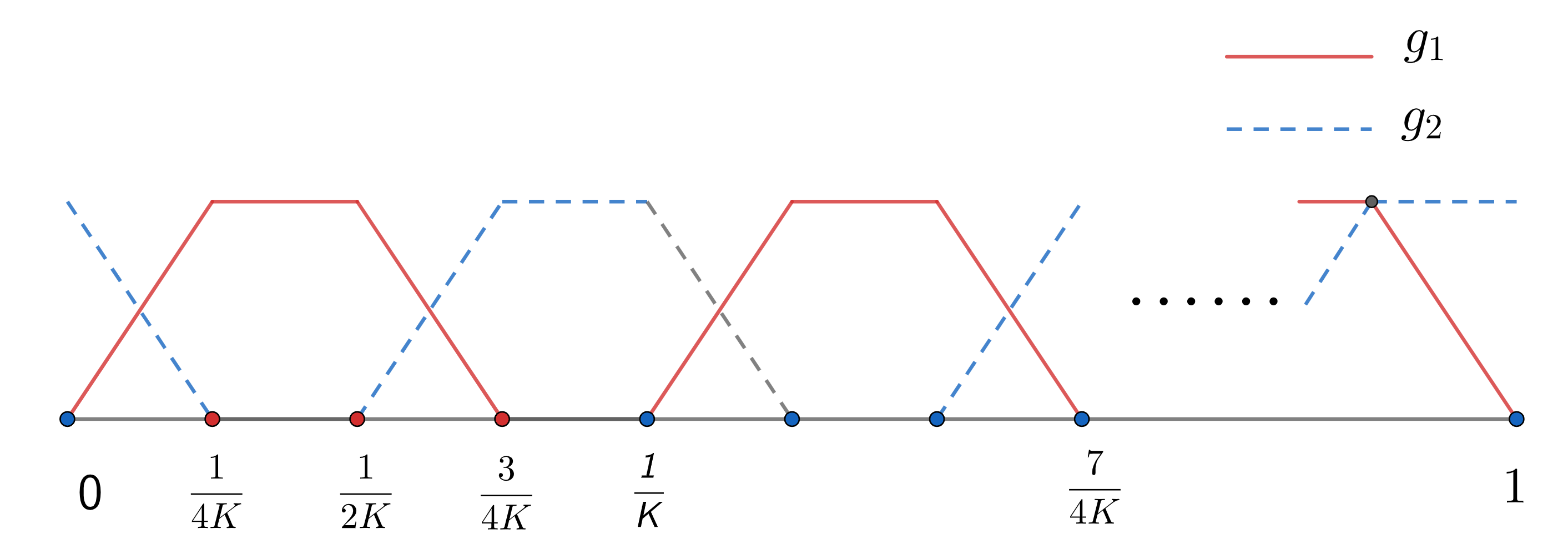}
			\caption{The schematic diagram of $g_i$ for $i=1,2$}
			\label{g12}
		\end{figure}
The following two lemmas and one proposition can be found in \cite{yang2023nearly}. For the readability in this paper, we present the proofs of these results here.

   \begin{lemma}\label{omegalem}For $\{\Omega_{\boldsymbol{m}}\}_{\boldsymbol{m}\in\{1,2\}^d}$ defined in Definition \ref{omega}, we have \[\bigcup_{\boldsymbol{m}\in\{1,2\}^d}\Omega_{\boldsymbol{m}}=[0,1]^d.\]\end{lemma}
		
		\begin{proof}
			We prove this lemma via induction. $d=1$ is valid due to $\Omega_1\cup\Omega_2=[0,1]$. Assume that the lemma is true for $d-1$, then \begin{align}
				\bigcup_{\boldsymbol{m}\in\{1,2\}^d}\Omega_{\boldsymbol{m}}=&[0,1]^d=\bigcup_{\boldsymbol{m}\in\{1,2\}^{d-1}}\Omega_{\boldsymbol{m}}\times \Omega_1+\bigcup_{\boldsymbol{m}\in\{1,2\}^{d-1}}\Omega_{\boldsymbol{m}}\times \Omega_2\notag\\=&\left([0,1]^{d-1}\times \Omega_1\right)\bigcup\left([0,1]^{d-1}\times \Omega_2\right)=[0,1]^d,
			\end{align} hence the case of $d$ is valid, and we finish the proof of the lemma.
		\end{proof}
		
		\begin{lemma}\label{prog}
			$\{g_{\boldsymbol{m}}\}_{\boldsymbol{m}\in\{1,2\}^d}$ defined in Definition \ref{gm} satisfies:
			
			(i): $\sum_{\boldsymbol{m}\in\{1,2\}^d}g_{\boldsymbol{m}}(\boldsymbol{x})=1$ for every $x\in[0,1]^d$.
			
			(ii): ${\rm supp}~g_{\boldsymbol{m}}\cap[0,1]^d\subset\Omega_{\boldsymbol{m}}$, where $\Omega_{\boldsymbol{m}}$ is defined in Definition \ref{omega}.
			
			(ii): For any $\boldsymbol{m}=(m_1,m_2,\ldots,m_d)\in\{1,2\}^d$ and $\boldsymbol{x}=(x_1,x_2,\ldots,x_d)\in[0,1]^d\backslash\Omega_{\boldsymbol{m}}$, there exists $j$ such as $g_{m_j}(x_j)=0$ and $\frac{\mathrm{d} g_{m_j}(x_j)}{\mathrm{d} x_j}=0$.
		\end{lemma}
		
		\begin{proof}
			(i) can be proved via induction as Lemma \ref{omegalem}, and we leave it to readers.
			
			As for (ii) and (iii), without loss of generality, we show the proof for $\boldsymbol{m}_*:=(1,1,\ldots,1)$. For any $\boldsymbol{x}\in [0,1]^d\backslash\Omega_{\boldsymbol{m}_*}$, there is $x_j\in [0,1]\backslash\Omega_1$. Then $g_1(x_j)=0$ and $g_{\boldsymbol{m}_*}(\boldsymbol{x})=\prod_{j=1}^d g_{1}(x_j)=0$, therefore ${\rm supp}~g_{\boldsymbol{m}_*}\cap[0,1]^d\subset\Omega_{\boldsymbol{m}_*}$. Furthermore, $\frac{\mathrm{d} g_{m_j}(x_j)}{\mathrm{d} x_j}=0$ for $x_j\in [0,1]\in\Omega_1$ due to the definition of $g_1$ (Definition \ref{gm}), then we finish this proof.
		\end{proof}

	Then we use the following proposition to approximate $\{g_{\boldsymbol{m}}\}_{\boldsymbol{m}\in\{1,2\}^d}$ by $\sigma$-NNs and construct a sequence of $\sigma$-NNs $\{\phi_{\boldsymbol{m}}\}_{\boldsymbol{m}\in\{1,2\}^d}$:
		
		\begin{proposition}\label{peri}
			Given any $N,L,n\in\sN_+$ for $K= N^2L^2$, then for any \[\boldsymbol{m}=(m_1,m_2,\ldots,m_d)\in\{1,2\}^d,\] there is a $\sigma$-NN with the width smaller than $(9+d)(N+1)+d-1$ and depth smaller than $15 d(d-1)n L$ such as \[\|\phi_{\boldsymbol{m}}(\boldsymbol{x})-g_{\boldsymbol{m}}(\boldsymbol{x})\|_{W^{1,\infty}([0,1]^d)}\le50 d^{\frac{5}{2}}(N+1)^{-4dnL}.\]
		\end{proposition}
  \begin{proof}
			The proof is similar to that of \cite[Proposition 1]{yang2023nearly}. For readability, we provide the proof here. First, we construct $g_1$ and $g_2$ by neural networks in $[0,1]$. We first construct a $\sigma$-NN in the small set $\left[0,NL^2\right]$. It is easy to check there is a neural network $\hat{\psi}$ with the width $4$ and one layer such as \begin{equation}
				\hat{\psi}(x):=
				\begin{cases}
					1,~&x\in \left[\frac{1}{8K},\frac{3}{8K}\right]  \\
					4K\left(x-\frac{1}{8K}\right),~&x\in\left[\frac{3}{8K},\frac{5}{8K}\right]\\
					-4K\left(x-\frac{7}{8K}\right),~&x\in\left[\frac{5}{8K},\frac{7}{8K}\right]\\
					0,~&\text{Otherwise}.
				\end{cases}
			\end{equation}
			
			Hence, we have a network $\psi_1$ with the width $4N$ and one layer such as \[\psi_1(x):=\sum_{i=0}^{N-1}\hat{\psi}\left(x-\frac{i}{K}\right).\]
			
			Next, we construct $\psi_i$ for $i=2,3,4$ based on the symmetry and periodicity of $g_i$. $\psi_2$ is the function with period $\frac{2}{NL^2}$ in $\left[0,\frac{1}{L^2}\right]$, and each period is a hat function with gradient 1. $\psi_3$ is the function with period $\frac{2}{L^2}$ in $\left[0,\frac{1}{L}\right]$, and each period is a hat function with gradient 1. $\psi_4$ is the function with period ${\frac{2}{L}}$ in $\left[0,1\right]$, and each period is a hat function with gradient 1. 
   
			Note that $\psi_2\circ\psi_3\circ\psi_4(x)$ is the function with period $\frac{2}{NL^2}$ in $[0,1]$, and each period is a hat function with gradient 1. Then function $\psi_1\circ \psi_2\circ\psi_3\circ\psi_4(x)$ is obtained by repeating reflection $\psi_1$ in $\left[0,1\right]$, which is the function we want.
			
			Similar with $\psi_1$, $\psi_2$ is a network with $4N$ width and one layer. Due to Proposition \ref{sum}, we know that $\psi_3$ and $\psi_4$ is a network with $7$ width and $L+1$ depth. Hence \begin{equation}
				\psi(x):=\psi_1\circ \psi_2\circ\psi_3\circ\psi_4(x)\label{rep}
			\end{equation} is a network with $4N$ width and $2L+4$ depth and $g_1=\psi\left(x+\frac{1}{8K}\right)$ and $g_1=\psi\left(x+\frac{5}{8K}\right)$.
			
			Now we can construct $g_{\boldsymbol{m}}$ for $m\in\{1,2\}^d$ based on Proposition \ref{prop1}: There is a neural network $\phi_\text{prod}$ with the width $9(N+1)+d-1$ and depth $14 d(d-1)n L$ such that $\|\phi_\text{prod}\|_{\mathcal{W}^{1, \infty}([0,1]^d)} \leq 18$ and
			$$
			\left\|\phi_\text{prod}(\boldsymbol{x})-x_1 x_2 \cdots x_d\right\|_{\mathcal{W}^{1, \infty}([0,1]^d)} \leq 10(d-1)(N+1)^{-7 d nL}.
			$$
			
			Then denote $\phi_{\boldsymbol{m}}(\boldsymbol{x}):=\phi_\text{prod}(g_{m_1},g_{m_2},\ldots,g_{m_d})$ which is a neural network with the width smaller than $(9+d)(N+1)+d-1$ and depth smaller than $15 d(d-1)n L$. Furthermore, due to Lemma \ref{composition}, we have \begin{align}
				\|\phi_{\boldsymbol{m}}(\boldsymbol{x})-g_{\boldsymbol{m}}(\boldsymbol{x})\|_{\mathcal{W}^{1, \infty}([0,1]^d)}\le& d^{\frac{3}{2}}\left\|\phi_\text{prod}(\boldsymbol{x})-x_1 x_2 \cdots x_d\right\|_{L^{ \infty}([0,1]^d)}\notag\\&+d^{\frac{3}{2}}\left\|\phi_\text{prod}(\boldsymbol{x})-x_1 x_2 \cdots x_d\right\|_{\mathcal{W}^{1, \infty}([0,1]^d)}|\psi|_{W^{1,\infty}(0,1)}\notag\\\le&d^{\frac{3}{2}}10(d-1)(N+1)^{-7 nd L}\left(1+4N^2L^2\right)\notag\\\le&50 d^{\frac{5}{2}}(N+1)^{-4dnL},	\end{align}
			where the last inequality is due to \[\frac{N^2L^2}{(N+1)^{ 3d nL}}\le \frac{N^2L^2}{(N+1)^{ 3d nL}}\le \frac{L^2}{(N+1)^{ 3d nL-2}}\le\frac{L^2}{2^{ d nL}}\le1.\]
		\end{proof}

  \begin{proposition}\label{main2}
			For any $f\in X^{2,\infty}(\Omega)$ with $p\ge1$ and $|f|_{2,\infty}\le 1$, $N, L\in\sN_+$, there is a $\sigma$-NN $\tilde{k}_{\boldsymbol{m}}(\boldsymbol{x})$ for any $\boldsymbol{m}\in\{1,2\}^d$ with $64d(N+2)(\log_2(8N))^{d+1}$  width and $(33L+2)(\log_2(4L))^{d+1}$ depth, such that\begin{align}
        \left\|\tilde{k}_{\boldsymbol{m}}(\boldsymbol{x})-f(\boldsymbol{x})\right\|_{H^1(\Omega_{\boldsymbol{m}})}\le  CN^{-2}L^{-2}.
   \end{align} where $C$ is the constant independent with $N,L$, and polynomially dependent on the dimension $d$.
		\end{proposition}

  \begin{proof}
      The proof is similar to that of Proposition \ref{main1}. We consider $\boldsymbol{m}=\textbf{1}$, i.e., $\Omega_{\boldsymbol{m}_*}=\Omega_\delta$. For other $\boldsymbol{m}\in\{1,2\}^d$, the proof can be carried out in a similar way. For any $|\boldsymbol{l}|\le n+d-1$, there exists a $\sigma$-NN $\psi_{\boldsymbol{l}}$ with $64d(N+1)\log_2(8N)$ width and $(33L+2)\log_2(4L)$ depth such that 
\begin{equation}
    \left\|\psi_{\boldsymbol{l}}(\boldsymbol{x})-\sum_{\boldsymbol{i} \in \boldsymbol{i}_{\boldsymbol{l}}} v_{\boldsymbol{l}, \boldsymbol{i}} \phi_{\boldsymbol{l}, \boldsymbol{i}}(\boldsymbol{x})\right\|_{W^{1,\infty}(\Omega_{\boldsymbol{m}_*})}\le (26+4C_{\alpha,\boldsymbol{l}_*})N^{-4}L^{-4}.
\end{equation} The proof follows a similar structure to that in Proposition \ref{main1}. This similarity arises from the fact that $\sum_{\boldsymbol{i} \in \boldsymbol{i}_{\boldsymbol{l}}} v_{\boldsymbol{l}, \boldsymbol{i}} \phi_{\boldsymbol{l}, \boldsymbol{i}}(\boldsymbol{x})=p_{\boldsymbol{l}}(\boldsymbol{x})q_{\boldsymbol{l}}(\boldsymbol{x})$ and $p_{\boldsymbol{l}}$ is a piece-wise constant function with a weak derivative always equal to zero. The approximation of $p_{\boldsymbol{l}}(\boldsymbol{x})$ has already been measured by the norm $W^{1,\infty}$ in Proposition \ref{main1}. Due to $W^{1,\infty}(\Omega)\subset H^1{1,\infty}(\Omega)$, we can have a $\sigma$-NN $\tilde{k}_{\boldsymbol{m}_*}(\boldsymbol{x})$ with 
\begin{equation}
    \text{width } 32d(N+1)\log_2(8N)(2\log_2(NL)+1)^d
\end{equation}
and 
\begin{equation}
    \text{depth } (33L+2)\log_2(4L)
\end{equation}
such that 
\begin{equation}
    \left\|\tilde{k}_{\boldsymbol{m}_*}(\boldsymbol{x})-f_1^{(n)}(\boldsymbol{x})\right\|_{H^1(\Omega_{\boldsymbol{m}_*})}\le (52+8C_{\alpha,\boldsymbol{l}_*})N^{-4}L^{-4}.
\end{equation}

Combine with the Lemma \ref{err-first}, we have \begin{align}
    \left\|\tilde{k}_{\boldsymbol{m}_*}(\boldsymbol{x})-f(\boldsymbol{x})\right\|_{H^1(\Omega_{\boldsymbol{m}_*})}\le &(52+8C_{\alpha,\boldsymbol{l}_*})N^{-4}L^{-4}+CN^{-2}L^{-2}\notag\\\le & CN^{-2}L^{-2}.
\end{align}
  \end{proof}

 Now we combine $\{\hat{k}_{\boldsymbol{m}}(\boldsymbol{x})\}_{\boldsymbol{m}\in\{1,2\}^d}$ and $\{\phi_{\boldsymbol{m}}(\boldsymbol{x})\}_{\boldsymbol{m}\in\{1,2\}^d}$ in Proposition \ref{peri} to extend the approximation into the whole domain $\Omega$. Before doing so, we require the following lemma. This lemma demonstrates that $\phi{\boldsymbol{m}}(\boldsymbol{x})$ in Proposition \ref{peri} attains $0$ to $0$ behavior in the Sobolev norms:
 \begin{lemma}\label{reduce}
			For any $\chi(\boldsymbol{x})\in H^1([0,1]^d)$, we have \begin{align}
				\|\phi_{\boldsymbol{m}}(\boldsymbol{x})\cdot\chi(\boldsymbol{x})\|_{H^1([0,1]^d)}=&	\|\phi_{\boldsymbol{m}}(\boldsymbol{x})\cdot\chi(\boldsymbol{x})\|_{H^1([0,1]^d)}\notag\\\|\phi_{\boldsymbol{m}}(\boldsymbol{x})\cdot\chi(\boldsymbol{x})-\phi(\phi_{\boldsymbol{m}}(\boldsymbol{x}),\chi(\boldsymbol{x}))\|_{H^1([0,1]^d)}=&	\|\phi_{\boldsymbol{m}}(\boldsymbol{x})\cdot\chi(\boldsymbol{x})-\phi(\phi_{\boldsymbol{m}}(\boldsymbol{x}),\chi(\boldsymbol{x}))\|_{H^1([0,1]^d)}
			\end{align} for any $\boldsymbol{m}\in\{1,2\}^d$, where $\phi_{\boldsymbol{m}}(\boldsymbol{x})$ and $\Omega_{\boldsymbol{m}}$ is defined in Proposition \ref{peri} and Definition.~\ref{omega}, and $\phi$ is from Proposition \ref{2prop} (choosing $a=1$ in the proposition).
		\end{lemma}

  \begin{proof}
			For the first equality, we only need to show that\begin{equation}
				\|\phi_{\boldsymbol{m}}(\boldsymbol{x})\cdot\chi(\boldsymbol{x})\|_{H^1\left([0,1]^d\backslash \Omega_{\boldsymbol{m}}\right)}=0.
			\end{equation}
			
			According to the Proposition \ref{peri}, we have $\phi_{\boldsymbol{m}}(\boldsymbol{x})=\phi_\text{prod}(g_{m_1},g_{m_2},\ldots,g_{m_d})$, and for any $\boldsymbol{x}=(x_1,x_2,\ldots,x_d)\in[0,1]^d\backslash \Omega_{\boldsymbol{m}}$, there is $m_j$ such as $g_{m_j}(x_j)=0$ and $\frac{d g_{m_j}(x_j)}{d x_j}=0$ due to Lemma \ref{prog}. Based on Eq.~(\ref{equal0}) in Proposition \ref{prop1}, we have \[\phi_{\boldsymbol{m}}(\boldsymbol{x})=\frac{\partial \phi_{\boldsymbol{m}}(\boldsymbol{x})}{\partial x_s}=0,~x\in [0,1]^d\backslash \Omega_{\boldsymbol{m}}, s\not=j.\]
			
			Furthermore, \begin{equation}
				\frac{\partial \phi_{\boldsymbol{m}}(\boldsymbol{x})}{\partial x_j}=\frac{\partial \phi_\text{prod}(g_{m_1},g_{m_2},\ldots,g_{m_d})}{\partial g_{m_j}}\frac{\mathrm{d}  g_{m_j}(x_j)}{\mathrm{d}  x_j}=0.
			\end{equation}
			Hence we have \begin{equation}
				|\phi_{\boldsymbol{m}}(\boldsymbol{x})\cdot\chi(\boldsymbol{x})|+\sum_{q=1}^d\left|\frac{\partial \left[\phi_{\boldsymbol{m}}(\boldsymbol{x})\cdot\chi(\boldsymbol{x})\right]}{\partial x_q}\right|=0
			\end{equation} for all $\boldsymbol{x}\in [0,1]^d\backslash \Omega_{\boldsymbol{m}}$.
			
			Similarly, for the second equality in this lemma, we have \begin{align}
				&|\phi(\phi_{\boldsymbol{m}}(\boldsymbol{x}),\chi(\boldsymbol{x}))|+\sum_{q=1}^d\left|\frac{\partial \left[\phi(\phi_{\boldsymbol{m}}(\boldsymbol{x}),\chi(\boldsymbol{x}))\right]}{\partial x_q}\right|\notag\\=&|\phi(0,\chi(\boldsymbol{x}))|+\sum_{q=1}^d\left[\left|\frac{\partial \left[\phi(0,\chi(\boldsymbol{x}))\right]}{\partial \chi(\boldsymbol{x})}\cdot\frac{\partial \chi(\boldsymbol{x})}{\partial x_q}\right|+\left|\frac{\partial \left[\phi(\phi_{\boldsymbol{m}}(\boldsymbol{x}),\chi(\boldsymbol{x}))\right]}{\partial \phi_{\boldsymbol{m}}(\boldsymbol{x})}\cdot \frac{\partial \phi_{\boldsymbol{m}}(\boldsymbol{x})}{\partial x_q}\right|\right]\notag\\=&0,
			\end{align} for all $\boldsymbol{x}\in [0,1]^d\backslash \Omega_{\boldsymbol{m}}$ based on \[\phi(0,y)=\frac{\partial \phi(0,y)}{\partial y}=0,~y\in(-M,M),\] and $\frac{\partial \phi_{\boldsymbol{m}}(\boldsymbol{x})}{\partial x_q}=0$. Hence we finish our proof.
		\end{proof}

  \begin{proof}[Proof of Theorem \ref{H1 K}]
				Based on Propositions \ref{main1} and \ref{main2}, there is a sequence of the neural network $\{\tilde{k}_{\boldsymbol{m}}(\boldsymbol{x})\}_{\boldsymbol{m}\in\{1,2\}^d}$ such that \begin{align}
        \left\|\tilde{k}_{\boldsymbol{m}}(\boldsymbol{x})-f(\boldsymbol{x})\right\|_{H^1(\Omega_{\boldsymbol{m}})}&\le  CN^{-2}L^{-2},\notag\\ \left\|\tilde{k}_{\boldsymbol{m}}(\boldsymbol{x})-f(\boldsymbol{x})\right\|_{L_2(\Omega)}&\le  CN^{-4}L^{-4}(\log_2N)^{d-1}(\log_2L)^{d-1},
   \end{align} where $C$ is independent with $N$ and $L$, and each $\tilde{k}_{\boldsymbol{m}}(\boldsymbol{x})$ for any $\boldsymbol{m}\in\{1,2\}^d$ is a $\sigma$-NN with $64d(N+2)(\log_2(8N))^{d+1}$  width and $(33L+2)(\log_2(4L))^{d+1}$ depth. According to Proposition \ref{peri}, there is a sequence of the neural network $\{\phi_{\boldsymbol{m}}(\boldsymbol{x})\}_{\boldsymbol{m}\in\{1,2\}^d}$ such that \[\|\phi_{\boldsymbol{m}}(\boldsymbol{x})-g_{\boldsymbol{m}}(\boldsymbol{x})\|_{W^{1,\infty}([0,1]^d)}\le50 d^{\frac{5}{2}}(N+1)^{-4dL},\]where $\{g_{\boldsymbol{m}}\}_{\boldsymbol{m}\in\{1,2\}^d}$ is defined in Definition \ref{gm} with $\sum_{\boldsymbol{m}\in\{1,2\}^d}g_{\boldsymbol{m}}(\boldsymbol{x})=1$ and ${\rm supp}~ g_{\boldsymbol{m}}\cap[0,1]^d=\Omega_{\boldsymbol{m}}$. For each $\phi_{\boldsymbol{m}}$, it is a neural network with the width smaller than $(9+d)(N+1)+d-1$ and depth smaller than $15 d(d-1) L$.
				
				Due to Proposition \ref{2prop}, there is a neural network $\widetilde{\Phi}$ with the width $15(N+1)$ and depth $14L$ such that $\|\phi\|_{W^{1,\infty}[-1,1]^2}\le 12$ and \begin{equation}
					\left\|\widetilde{\Phi}(x,y)-xy\right\|_{W^{1,\infty}[-1,1]^2}\le 6(N+1)^{-7(L+1)}.
				\end{equation}
				
				Now we define \begin{equation}
					k(\boldsymbol{x})=\sum_{\boldsymbol{m}\in\{1,2\}^d}\phi(\phi_{\boldsymbol{m}}(\boldsymbol{x}),\tilde{k}_{\boldsymbol{m}}(\boldsymbol{x})).
				\end{equation}
				
				Note that \begin{align}
					\fR:=&\|f(\boldsymbol{x})-\phi(\boldsymbol{x})\|_{H^1([0,1]^d)}=\left\|\sum_{\boldsymbol{m}\in\{1,2\}^d} g_{\boldsymbol{m}}\cdot f(\boldsymbol{x})-\phi(\boldsymbol{x})\right\|_{H^1([0,1]^d)}\notag\\\le &\left\|\sum_{\boldsymbol{m}\in\{1,2\}^d} \left[g_{\boldsymbol{m}}\cdot f(\boldsymbol{x})-\phi_{\boldsymbol{m}}(\boldsymbol{x})\cdot\psi_{\boldsymbol{m}}(\boldsymbol{x})\right]\right\|_{H^1([0,1]^d)}\notag\\&+\left\|\sum_{\boldsymbol{m}\in\{1,2\}^d} \left[\phi_{\boldsymbol{m}}(\boldsymbol{x})\cdot\psi_{\boldsymbol{m}}(\boldsymbol{x})-\widetilde{\Phi}(\phi_{\boldsymbol{m}}(\boldsymbol{x}),\psi_{\boldsymbol{m}}(\boldsymbol{x}))\right]\right\|_{H^1([0,1]^d)}.
				\end{align}
				
				As for the first part, \begin{align}
					&\left\|\sum_{\boldsymbol{m}\in\{1,2\}^d} \left[g_{\boldsymbol{m}}\cdot f(\boldsymbol{x})-\phi_{\boldsymbol{m}}(\boldsymbol{x})\cdot\psi_{\boldsymbol{m}}(\boldsymbol{x})\right]\right\|_{H^1([0,1]^d)}\notag\\\le& \sum_{\boldsymbol{m}\in\{1,2\}^d}\left\| g_{\boldsymbol{m}}\cdot f(\boldsymbol{x})-\phi_{\boldsymbol{m}}(\boldsymbol{x})\cdot\psi_{\boldsymbol{m}}(\boldsymbol{x})\right\|_{H^1([0,1]^d)}\notag\\\le &\sum_{\boldsymbol{m}\in\{1,2\}^d}\left[\left\| (g_{\boldsymbol{m}}-\phi_{\boldsymbol{m}}(\boldsymbol{x}))\cdot f(\boldsymbol{x})\right\|_{H^1([0,1]^d)}+\left\| (f_{\boldsymbol{m}}-\psi_{\boldsymbol{m}}(\boldsymbol{x}))\cdot \phi_{\boldsymbol{m}}(\boldsymbol{x})\right\|_{H^1([0,1]^d)}\right]\notag\\=&\sum_{\boldsymbol{m}\in\{1,2\}^d}\left[\left\| (g_{\boldsymbol{m}}-\phi_{\boldsymbol{m}}(\boldsymbol{x}))\cdot f(\boldsymbol{x})\right\|_{H^1([0,1]^d)}+\left\| (f_{\boldsymbol{m}}-\psi_{\boldsymbol{m}}(\boldsymbol{x}))\cdot \phi_{\boldsymbol{m}}(\boldsymbol{x})\right\|_{H^1([0,1]^d)}\right],
				\end{align}where the last equality is due to Lemma \ref{reduce}. Based on $\|f\|_{W^{1,\infty}([0,1]^d)}\le 1$, we have \begin{align}
					\left\| (g_{\boldsymbol{m}}-\phi_{\boldsymbol{m}}(\boldsymbol{x}))\cdot f(\boldsymbol{x})\right\|_{H^1([0,1]^d)}\le \left\| (g_{\boldsymbol{m}}-\phi_{\boldsymbol{m}}(\boldsymbol{x}))\right\|_{H^1([0,1]^d)}\le 50 d^{\frac{5}{2}}(N+1)^{-4dnL}.
				\end{align} And \begin{align}
					&\left\| (f_{\boldsymbol{m}}-\psi_{\boldsymbol{m}}(\boldsymbol{x}))\cdot \phi_{\boldsymbol{m}}(\boldsymbol{x})\right\|_{H^1([0,1]^d)}\notag\\\le& \left\| (f_{\boldsymbol{m}}-\psi_{\boldsymbol{m}}(\boldsymbol{x}))\right\|_{H^1([0,1]^d)}\cdot \|\phi_{\boldsymbol{m}}\|_{L_\infty(\Omega_{\boldsymbol{m}})}+\left\| (f_{\boldsymbol{m}}-\psi_{\boldsymbol{m}}(\boldsymbol{x}))\right\|_{L^{2}(\Omega_{\boldsymbol{m}})}\cdot \|\phi_{\boldsymbol{m}}\|_{W^{1,\infty}(\Omega_{\boldsymbol{m}})}\notag\\\le &CN^{-2}L^{-2}\cdot\left( 1+50d^{\frac{5}{2}}\right)+CN^{-4}L^{-4}\cdot54d^{\frac{5}{2}}{ N^2 L^2}(\log_2N)^{d-1}(\log_2L)^{d-1}\notag\\\le& C{ N^{-2} L^{-2}}(\log_2N)^{d-1}(\log_2L)^{d-1},
				\end{align}
				where the second inequality is due to \begin{align}
					&\|\phi_{\boldsymbol{m}}\|_{L_\infty(\Omega_{\boldsymbol{m}})}\le \|\phi_{\boldsymbol{m}}\|_{L_\infty([0,1]^d)}\le \|g_{\boldsymbol{m}}\|_{L_\infty([0,1]^d)}+\|\phi_{\boldsymbol{m}}-g_{\boldsymbol{m}}\|_{L_\infty([0,1]^d)}\le 1+50d^{\frac{5}{2}}\notag\\
					&\|\phi_{\boldsymbol{m}}\|_{W^{1,\infty}(\Omega_{\boldsymbol{m}})}\le \|\phi_{\boldsymbol{m}}\|_{W^{1,\infty}([0,1]^d)}\le \|g_{\boldsymbol{m}}\|_{W^{1,\infty}([0,1]^d)}+\|\phi_{\boldsymbol{m}}-g_{\boldsymbol{m}}\|_{W^{1,\infty}([0,1]^d)}\notag\\&\le 4N^2L^2+50d^{\frac{5}{2}}.
				\end{align}
				
				Therefore\begin{align}
					\left\|\sum_{\boldsymbol{m}\in\{1,2\}^d} \left[g_{\boldsymbol{m}}\cdot f(\boldsymbol{x})-\phi_{\boldsymbol{m}}(\boldsymbol{x})\cdot\psi_{\boldsymbol{m}}(\boldsymbol{x})\right]\right\|_{W^{1,\infty}([0,1]^d)}\le C{ N^{-2} L^{-2}}(\log_2N)^{d-1}(\log_2L)^{d-1},\label{r1}
				\end{align} due to $(N+1)^{-4dnL}\le N^{-2n}L^{-2n}$.
				
				For the second part, due to Lemma \ref{reduce}, we have \begin{align}&\left\|\sum_{\boldsymbol{m}\in\{1,2\}^d} \left[\phi_{\boldsymbol{m}}(\boldsymbol{x})\cdot\psi_{\boldsymbol{m}}(\boldsymbol{x})-\widetilde{\Phi}(\phi_{\boldsymbol{m}}(\boldsymbol{x}),\psi_{\boldsymbol{m}}(\boldsymbol{x}))\right]\right\|_{H^1([0,1]^d)}\notag\\\le&\sum_{\boldsymbol{m}\in\{1,2\}^d}\left\| \phi_{\boldsymbol{m}}(\boldsymbol{x})\cdot\psi_{\boldsymbol{m}}(\boldsymbol{x})-\widetilde{\Phi}(\phi_{\boldsymbol{m}}(\boldsymbol{x}),\psi_{\boldsymbol{m}}(\boldsymbol{x}))\right\|_{H^1([0,1]^d)}\notag\\=&\sum_{\boldsymbol{m}\in\{1,2\}^d}\left\| \phi_{\boldsymbol{m}}(\boldsymbol{x})\cdot\psi_{\boldsymbol{m}}(\boldsymbol{x})-\widetilde{\Phi}(\phi_{\boldsymbol{m}}(\boldsymbol{x}),\psi_{\boldsymbol{m}}(\boldsymbol{x}))\right\|_{H^1(\Omega_{\boldsymbol{m}})}.
				\end{align}
				
			Due to Lemma \ref{composition}, we have that \begin{align}
					&\left\| \phi_{\boldsymbol{m}}(\boldsymbol{x})\cdot\psi_{\boldsymbol{m}}(\boldsymbol{x})-\widetilde{\Phi}(\phi_{\boldsymbol{m}}(\boldsymbol{x}),\psi_{\boldsymbol{m}}(\boldsymbol{x}))\right\|_{H^1(\Omega_{\boldsymbol{m}})}\le C{ N^{-2} L^{-2}}(\log_2N)^{d-1}(\log_2L)^{d-1}.\label{r2}
				\end{align}
				
				Combining (\ref{r1}) and (\ref{r2}), we have that there is a $\sigma$-NN with $(47L+2)(\log_2(4L))^{d+1}$ depth and $2^{d+6}d(N+2)(\log_2(8N))^{d+1}$  width such that\[\|f(\boldsymbol{x})-k(\boldsymbol{x})\|_{H^{1}([0,1]^d)}\le C{ N^{-2} L^{-2}}(\log_2N)^{d-1}(\log_2L)^{d-1},\]where $C$ is the constant independent with $N,L$.
			\end{proof}

   The approximation rate for Korobov spaces provided in Theorem \ref{H1 K} falls short of achieving the nearly optimal approximation rate observed in function spaces $W^{2d,p}$, as measured by the norm containing the first derivative \cite{yang2023nearly}. In the latter case, the optimal rate is $\fO((NL)^{-\frac{4d-2}{d}})$. The limitation in achieving this optimal rate for Korobov spaces is rooted in the sensitivity of these functions to derivatives. For instance, consider a finite expansion of $f$ in Korobov spaces denoted as $f_n^{(1)}(\boldsymbol{x})=\sum_{|\boldsymbol{l}|_1 \leq n+d-1} \sum_{\boldsymbol{i} \in \boldsymbol{i}_{\boldsymbol{l}}} v_{\boldsymbol{l}, \boldsymbol{i}} \phi_{\boldsymbol{l}, \boldsymbol{i}}(\boldsymbol{x})$. In this expansion, there exists a spline function $\phi_{\boldsymbol{l}, \boldsymbol{i}}(\boldsymbol{x})$ for which $\boldsymbol{l}=(n,1,1,1,\ldots,1)$, and its partial derivative with respect to $x_1$ can be very large, on the order of $2^n$. The way to prove the optimality of the $H^1$ case is similar to Theorem \ref{Optimality} and combined with the following lemma:
   \begin{lemma}[{\cite[Theorem 1]{yang2023nearly}}]\label{vcdim1}
		For any $N,L,d\in\sN_+$, there exists a constant $\bar{C}$ independent with $N,L$ such that	\begin{equation}
		\text{VCdim}(D\Phi)\le \bar{C} N^2L^2\log_2 L\log_2 N,\label{bound}\end{equation}for \begin{align}
			D\Phi:=\left\{\psi=D_i\phi:\phi\in\Phi,~i=1,2,\ldots,d\right\},\end{align}where $ \Phi:=\left\{\phi:\phi\text{ is a $\sigma$-NN in $\sR^d$ with width$\le N$ and depth$\le L$}\right\}$, and $D_i$ is the weak derivative in the $i$-th variable.
	\end{lemma}
 \begin{theorem}\label{OptimalityH1}
			Given any $\rho, C_{1}, C_{2}, C_{3}, J_0>0$ and $n,d\in\sN^+$, there exist $N,L\in\sN$ with $NL\ge J_0$ and $f\in X^{2,\infty}$ with $|f|_{2,\infty}\le 1$, such that\begin{equation}
				\inf_{\phi\in\fK}\|\phi-f\|_{H^1(\Omega)}> C_{3}L^{-2-\rho}N^{-2-\rho},
			\end{equation} where \[\fK:=\{\text{$\sigma$-NNs in $\sR^d$ with the width $C_{1} N(\log_2 N)^{d+1}$ and depth $C_{2} L(\log_2 L)^{d+1}$}\}.\]
		\end{theorem}
  \begin{proof}
     The proof is similar to Theorem \ref{Optimality} and combined with Lemma \ref{vcdim1}.
  \end{proof}

Comparing the results in Theorem \ref{H1} and Corollary \ref{L2} with Theorems \ref{L2 K} and \ref{H1 K}, we observe that the results in Theorems \ref{L2 K} and \ref{H1 K} are significantly better than those in Theorem \ref{H1} and Corollary \ref{L2}. The constants in Theorems \ref{H1} and \ref{L2} are superior to those in Theorems \ref{L2 K} and \ref{H1 K}, which exponentially depend on the dimension $d$. This leaves an open question for future research to explore alternative approaches for addressing the challenge of incorporating the dependence on $d$ in the lower bounds while maintaining a \textit{super-convergence} rate.
   \section{Conclusion}
This paper establishes the approximation of deep neural networks (DNNs) for Korobov spaces, not only in $L_p$ norms for $2 \le p \le \infty$ but also in $H^1$ norms, effectively avoiding the curse of dimensionality. For both types of errors, we establish a \textit{super-convergence} rate and prove the optimality of each approximation. %Additionally, we derive the generalization error for Korobov spaces using $H^1$ loss functions.

In our exploration of deep neural networks for approximating Korobov spaces, we note that prior work, such as \cite{blanchard2021shallow}, has focused on two-hidden layer neural networks for shallow approximations. The establishment of the potential of one-hidden layer neural networks for approximating functions in Korobov spaces is considered as future work. Moreover, in this paper, we delve into proving the optimality of our results. The proof strategy relies on the fact that the approximation rate in $X^{2,\infty}([0,1]^d)$ achieves a nearly optimal approximation rate for $W^{2,\infty}([0,1])$. However, when combining our work with the estimates provided in \cite{mao2022approximation}, it becomes evident that the \textit{super-convergence} rate for $X^{2,p}$ can only achieve $\mathcal{O}\left(N^{-4+\frac{2}{p}}L^{-4+\frac{2}{p}}\right)$ (up to logarithmic factors). Determining whether this rate is nearly optimal and establishing a proof for it remains an open question.

\section*{Acknowledgment}

YL is supported by the National Science Foundation through the award DMS-2343135.

\bibliographystyle{plain}
\bibliography{references}

\end{document}